\documentclass[12pt]{elsarticle}

\makeatletter
\def\ps@pprintTitle{%
 \let\@oddhead\@empty
 \let\@evenhead\@empty
 \def\@oddfoot{}%
 \let\@evenfoot\@oddfoot}
\makeatother

\usepackage[utf8]{inputenc}

\usepackage{fullpage}

\usepackage{amsmath}
\usepackage{amssymb}
\usepackage{latexsym}
\usepackage{amsthm}
\usepackage[retainorgcmds]{IEEEtrantools}

\usepackage{geometry}
\usepackage{graphicx}

\newtheorem{theorem}{Theorem}
\newtheorem{lemma}[theorem]{Lemma}
\newtheorem{corollary}[theorem]{Corollary}
\newtheorem{proposition}[theorem]{Proposition}

\usepackage{minitoc}

\usepackage{ifpdf}

\newcommand{\ucon}{\textnormal{con}^+}
\newcommand{\lcon}{\textnormal{con}^-}
\newcommand{\con}{\textnormal{con}}

\newtheorem{nclaim}{Claim}{\bf}{}
\newtheorem{observation}{Observation}{\bf}{}

\begin{document}

\begin{frontmatter}



\title{The strong convexity spectra of grids\tnoteref{t1}}

\tnotetext[t1]{This research was supported by CONACyT-M\'exico under Projects 178395 and 166306 and PAPIIT-M\'exico under Projects IN101912 and IN104915.}

\author[IMATE]{Gabriela Araujo-Pardo}
\ead{garaujo@matem.unam.mx}
\author[IMATE,FC]{C\'esar Hern\'andez-Cruz\corref{cor1}}
\ead{cesar@matem.unam.mx}
\author[IMATE]{Juan Jos\'e Montellano-Ballesteros}
\ead{juancho@matem.unam.mx}

\address[IMATE]{Instituto de Matem\'aticas, Universidad Nacional Aut\'onoma de M\'exico, \\ Area de la Investigaci\'on Cient\'ifica, Circuito Exterior, C.U. \\Coyoac\'an, C.P. 04510, M\'exico, D.F.}
\address[FC]{Facultad de Ciencias, Universidad Nacional Aut\'onoma de M\'exico, \\ Ciudad Universitaria, C.P. 04510,  M\'exico, D.F.}

\cortext[cor1]{Corresponding author, Tel. (+52)(55)56225427}

\begin{abstract}
Let $D$ be a connected oriented graph. A set $S \subseteq V(D)$ is convex in $D$ if, for every pair of vertices $x, y \in S$, the vertex set of every $xy$-geodesic, ($xy$ shortest directed path) and every $yx$-geodesic in $D$ is contained in $S$. The convexity number, $\con(D)$, of a non-trivial oriented graph, $D$, is the maximum cardinality of a proper convex set of $D$.   The strong convexity spectrum of the graph $G$, $S_{SC} (G)$, is the set $\{\con(D) \colon\ D \textnormal{ is a strong orientation of G} \}$.    In this paper we prove that the problem of determining the convexity number of an oriented graph is $\mathcal{NP}$-complete, even for bipartite oriented graphs of arbitrary large girth, extending previous known results for graphs.   We also determine $S_{SC} (P_n \Box P_m)$, for every pair of integers $n,m \ge 2$.
\end{abstract}


\end{frontmatter}

\section{Introduction}
\label{intro}

Graphs considered in the paper are finite, without loops or multiple edges. In a graph $G = (V , E)$, $V$ and $E$ ($V(G)$ and $E(G)$) denote the vertex set and the edge set of $G$, respectively.    For undefined concepts and notation we refer the reader to \cite{BJD}.

For two vertices $u$ and $v$ in a graph $G$, a $uv$-geodesic is a shortest path between $u$ and $v$.   A set $S$ of vertices of $G$ is convex if the vertices of every $uv$-geodesic is contained in $S$ for every $u, v \in S$.   According to Duchet, convexity in graphs has been studied since the early seventies, when abstract convexity was studied in different contexts (\cite{D} is an outdated, but very nice, survey on the subject).   Convexity in graphs has taken many different directions, and different related parameters have been defined and widely studied, e.g., the hull number \cite{ES}, the geodetic number \cite{HLT}, and the convexity number \cite{CWZ} of a graph.   Recent papers on this subjects include \cite{ACGNSS,DPRS2,DPRS}, where the decision problem associated with these three parameters are shown to be $\mathcal{NP}$-complete, even when restricted to bipartite graphs, and in the case of the geodetic number, even when restricted to bipartite chordal graphs.

Chartrand, Fink and Zhang generalized the concept of convexity to oriented graphs, and defined the convexity number for an oriented graph; oriented analogues of the hull number and geodetic number are defined in \cite{CZ}.   We focus on the convexity number of oriented graphs; although this generalization was introduced in 2002, and the proof given by Gimbel in \cite{G} on the $\mathcal{NP}$-completeness of determining the convexity number of an arbitrary graph is one of the shortest and neatest $\mathcal{NP}$-completeness proofs ever done, the problem of determining the convexity number of an oriented graph was not known to be $\mathcal{NP}$-complete until now.   We prove that determining the convexity number of an oriented graph is $\mathcal{NP}$-complete even when restricted to bipartite graphs of girth $g$,  with $g \ge 6$.

An \emph{oriented graph} is an orientation of some graph. In an oriented graph $D = (V , E)$, $V$ and $E$ ($V(D)$ and $E(D)$) denote the vertex set and the edge set of $D$, respectively. An \emph{oriented subgraph} $D' = (V', E')$ of an oriented graph $D = (V , E)$ is an oriented graph with $V' \subseteq V$ and $E' \subseteq E$. An oriented graph is \emph{connected} if its underlying graph is connected. A \emph{directed path} is a sequence $(v_1,v_2,...,v_k)$ of vertices of an oriented graph $D$ such that $v_1,v_2,...,v_k$ are distinct and $(v_i,v_{i+1}) \in E(D)$ for $i \in \{ 1, 2, . . . , k-1 \}$.   An oriented graph is \emph{strongly connected} (or \emph{strong}) if for every pair of distinct vertices $u$ and $v$, there exists a directed path from $u$ to $v$.   The \emph{girth} of an oriented graph is the length of a shortest directed cycle.

A $uv$-\emph{geodesic} in a digraph $D$ is a shortest $uv$-directed path and its length is $d_D (u, v)$.   A nonempty subset, $S$, of the vertex set of a digraph, $D$, is called a \emph{convex set} of $D$ if, for every $u, v \in S$, every vertex lying on a $uv$- or $vu$-geodesic belongs to $S$. For a nonempty subset, $A$, of $V(D)$, the \emph{convex hull}, $[A]$, is the minimal convex set containing $A$. Thus $[S] = S$ if and only if $S$ is convex in $D$. The \emph{convexity number}, $\con(D)$, of a digraph $D$ is the maximum cardinality of a proper convex set of $D$. A \emph{maximum convex set} $S$, of a digraph $D$, is a convex set with cardinality $\con(D)$. Since every singleton vertex set is convex in a connected oriented graph $D$, $1\le \con(D) \le n-1$. The \emph{degree}, $d(v)$, of a vertex $v$ in an oriented graph is the sum of its in-degree and out-degree; this is, $d(v) = d^-(v) + d^+(v)$.   A vertex, $v$, is an \emph{end-vertex} if $d(v) = 1$. A \emph{source} is a vertex having positive out-degree and in-degree $0$, while a \emph{sink} is a vertex having positive in-degree and out-degree $0$. For a vertex $v$ of $D$, the in-neighborhood of $v$, $N^-(v)$, is the set $\{ x \colon\ (v, x) \in E(D) \}$ and the out-neighborhood of $v$, $N^+ (v)$, is the set $\{ x \colon\ (x, v) \in E(D)\}$.   A vertex $v$ of $D$ is a \emph{transitive vertex} if $d^+(v)>0$, $d^-(v) > 0$ and, for every $u \in N^+(v)$ and $w \in N^-(v)$, $(w, u) \in E(D)$.

For graphs $G$ and $H$, their {\em cartesian product}, $G \Box H$, is the graph with vertex set $V(G) \times V(H)$, and such that two vertices $(g_1,h_1)$ and $(g_2,h_2)$ are adjacent in $G \Box H$ if either $g_1=g_2$ and $h_1 h_2$ is an edge in $H$, or $h_1=h_2$ and $g_1 g_2$ is an edge in $G$.   For a vertex $g$ of $G$, the subgraph of $G \Box H$ induced by the set $\{ (g,h) \colon\ h \in H \}$ is called an $H$-fiber and is denoted by $^gH$. Similarly, for $h \in H$, the $G$-fiber, $G^h$, is the subgraph induced by $\{ (g,h) \colon g \in G \}$. We will have occasion to use the fiber notation $G^h$ and $^gH$ to refer instead to the set of vertices in these subgraphs; the meaning will be clear from the context. It is clear that all $G$-fibers are isomorphic to $G$ and all $H$-fibers are isomorphic to $H$.

As mentioned, the concept of convexity number of an oriented graph was first introduced by Chartrand, Fink and Zhang in \cite{CFZ}, where they proved the following pair of theorems.

\begin{theorem}
Let $D$ be a connected oriented graph of order $n \ge 2$.   Then $\con (D) = n-1$ if and only if $D$ contains a source, a sink or a transitive vertex.
\end{theorem}

\begin{theorem} \label{no2}
There is no connected graph of order at least $4$ with convexity number $2$.
\end{theorem}

Taking an interesting direction for the subject of convexity in oriented graphs, in \cite{TYF}, Tong, Yen and Farrugia introduced the concepts of convexity spectrum and strong convexity spectrum of a graph.   For a nontrivial connected graph $G$, we define the \emph{convexity spectrum}, $S_C (G)$, of a graph $G$, as the set of convexity numbers of all orientations of $G$, and the \emph{strong convexity spectrum}, $S_{SC}(G)$, of a graph $G$ as the set of convexity numbers of all strongly connected orientations of $G$.   If $G$ has no strongly connected orientation, then $S_{SC}(G)$ is empty.   The \emph{lower orientable convexity number}, $\lcon (G)$, of $G$ is defined to be $\min S_C(G)$ and the \emph{upper orientable convexity number}, $\ucon(G)$, is defined to be $\max S_C(G)$.   Hence, for every nontrivial connected graph $G$ of order $n$, $1\le \lcon(G) \le \ucon(G) \le n-1$.

Tong, Yen and Farrugia calculated the convexity and strong convexity spectra of complete graphs and also constructed, for every $a \in \mathbb{Z}^+$, a graph $G$ with convexity spectrum $\{ a, n-1 \}$.   It is not very surprising that the strong convexity spectra of $K_n$ for $n \ge 7$ is a ``large'' set, $\{ 1, 3, 5,6, \dots, n-2 \}$, missing only $2, 4$ and $n-1$.   Nonetheless, we find very surprising that, in one hand, Tong and Yen proved in \cite{TY} that $S_{SC} (K_{r,s}) = \{ 1 \}$  for every pair of integers $2 \le r, s$, and, in the other hand, we prove that the strong convexity spectrum of an $n \times m$ grid, for any pair of integers $n,m \ge 5$, only lacks the set of integers $\{ 2, 3, 5, n-1, n-2, n-3, n-4, n-5 \}$.   So, an interesting question arises from the previous observation: What property in a graph determines a large strong convexity spectrum?   We can discard regularity and high degrees; grids are not regular graphs, and have both small maximum and minimum degree.   Although we did not find what is so special about grids in terms of convexity, we managed to calculate the strong convexity spectra of all grids.

The rest of this paper is ordered as follows.

Section \ref{SNP} is devoted to prove the $\mathcal{NP}$-completeness of the problem of determining the convexity number of a given oriented graph; the problem remains $\mathcal{NP}$-complete even when restricted to bipartite oriented graphs of arbitrarily large girth.   In Section \ref{SGrids}, we prove some basic results on the convexity number of general oriented graphs, and also introduce a concept of main importance to this work: The whirlpool orientation of a grid.   Using this concept, we prove that $1 \in S_{SC} (G)$ for any grid $G$.   We finish the section with a result excluding some values from the strong convexity spectra of certain grids.    Section \ref{SCSSG} is devoted to calculate the strong convexity spectra of $n \times 2$ and $n \times 3$ grids for every integer $n \ge 2$.  In Section \ref{SMain}, the strong convexity spectra of $n \times m$ grids for every pair of integers $n, m \ge 4$ is calculated.

\section{$\mathcal{NP}$-completeness}
\label{SNP}

We define the problem \textsc{Oriented Convexity Number} as follows.   Given an ordered pair $(D,k)$, consisting of an oriented graph $D$ and a positive integer $k$, determine whether $D$ has a convex set of size at least $k$.

This first section is devoted to prove the $\mathcal{NP}$-completeness of \textsc{Oriented Convexity Number}.

\begin{theorem}\label{redu}
\textsc{Oriented Convexity Number} restricted to bipartite oriented graphs of girth $6$ is $\mathcal{NP}$-complete.
\end{theorem}

\begin{proof}
Let $D$ be an oriented graph.   Given a subset $C \subseteq V(D)$, it can be verified in polynomial time whether $C$ is a convex set.   Hence, \textsc{Oriented Convexity Number} is in $\mathcal{NP}$.

In order to prove $\mathcal{NP}$-hardness (and hence $\mathcal{NP}$-completness), we reduce an instance $(G,k)$ of the well-known $\mathcal{NP}$-complete problem \textsc{Clique} to an instance $(D,k')$ of \textsc{Oriented Convexity Number} such that $\omega(G) \ge k$ if and only if $\con(D) \ge k'$, the encoding length of $(D,k')$ is polinomially bounded in terms of the encoding length of $(G,k)$, and $D$ is bipartite.

Let $(G,k)$ be an instance of \textsc{Clique}.   Let us assume that $k \ge 3$ and $G$ is connected; we construct $D$ as follows.   For every vertex $u$ of $G$ we create a directed hexagon, $H_u$, with two antipodal distinguished vertices $x_u$ and $y_u$.   For every edge $uv \in E(G)$ we add the arcs $(x_u,y_v)$ and $(x_v,y_u)$ to $D$.   We create four additional vertices $z_1, z_2, z_3, z_4$ with arcs $(z_i, z_{i+1})$ for $1 \le i \le 3$ and arcs $(x_u, z_1), (z_4, y_u)$ for every $u \in V(G)$.

\begin{figure}
\begin{center}
\includegraphics[width=\textwidth]{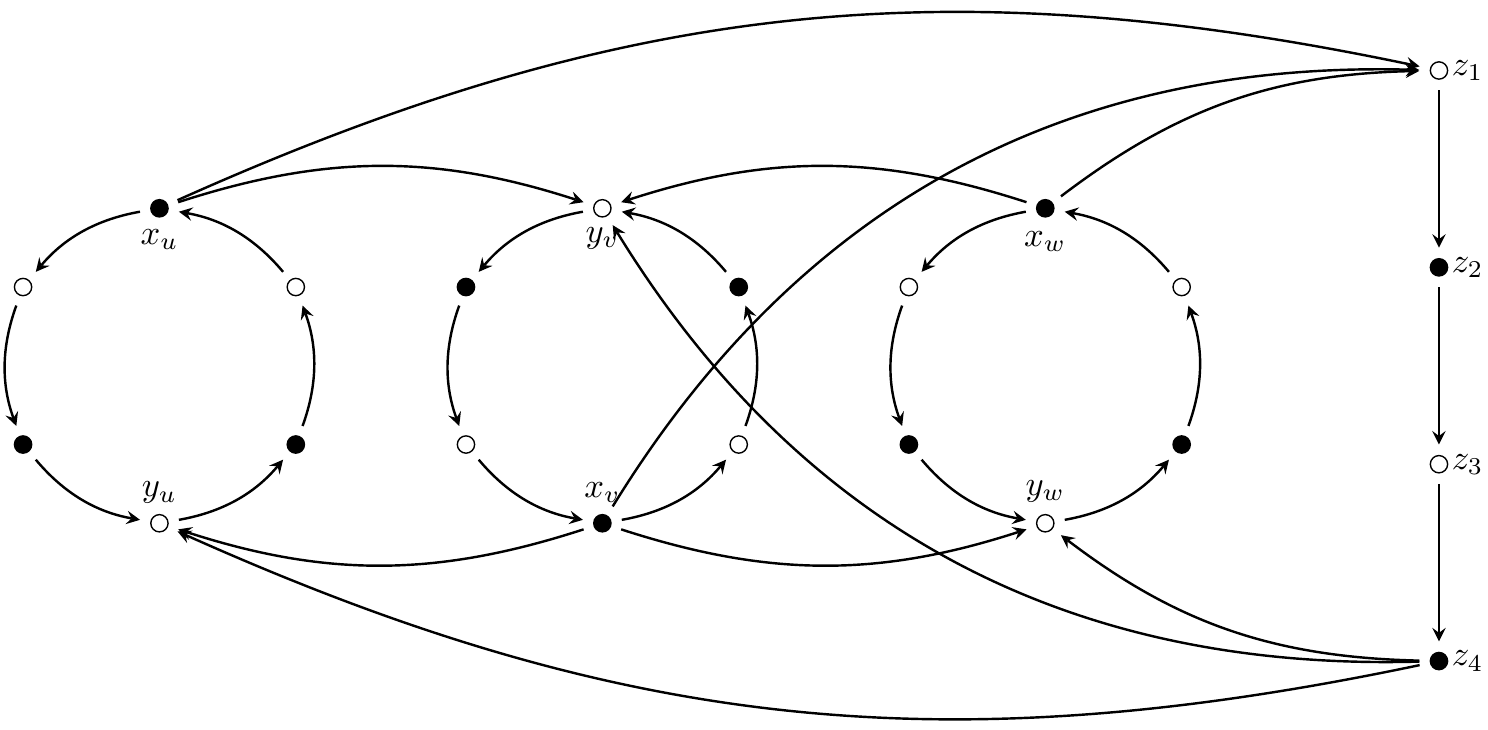}
\caption{The digraph $D$ of Theorem \ref{redu} when $G$ is $P_3 = (u,v,w)$.} \label{reduFig}
\end{center}
\end{figure}

Clearly, $|V(D)| = 6|V(G)| + 4$ and $|A(D)| = 8|V(G)| + 2|E(G)|+ 3$.   It is direct to observe that the digraph $D$ is bipartite and strongly connected.   In Figure \ref{reduFig} the classes of the bipartition are given by the vertices of the same color (black and white).   We also claim the following statements to hold.

\begin{nclaim} \label{C1}
Let $u$ be a vertex in $G$ and let $C$ be a convex set of $D$ with $|C| \ge 2$.   If $C \cap V(H_u) \ne \varnothing$, then $V(H_u) \subseteq C$.
\end{nclaim}

\begin{nclaim} \label{C2}
Let $C$ be a convex set of $D$ with $|C| \ge 2$.  If $z_i \in C$ for $1 \le i \le 4$, then $C = V(D)$.
\end{nclaim}

\begin{nclaim}\label{C3}
Let $u, v \in V(G)$ be such that $d_G (u,v) \ge 2$.   If $C$ is a convex set of $D$ such that $C \cap V(H_u) \ne \varnothing \ne C \cap V(H_v)$, then $C=V(D)$.
\end{nclaim}

\begin{nclaim} \label{C4}
If $S$ is a clique of $G$, then $C=\bigcup_{v \in S} V(H_v)$ is a convex set of $D$.
\end{nclaim}

It follows from Claims \ref{C1}-\ref{C4} that $C$ is a proper convex set of $D$ with $|C| \ge 2$, if and only if there exists a clique $S$ in $G$ such that $C=\bigcup_{v \in S} V(H_v)$.   Considering a maximum clique of $G$, and a maximum proper convex set of $D$, we obtain $\omega(G) = 6 \con(D)$.   Hence, if $k'=6k$, then $G$ contains a clique of size at least $k$ if and only if $D$ contains a convex set of size at least $k'$.   Since the encoding length of $(D,k')$ is linearly bounded in terms of the encoding length of $(G,k)$,  \textsc{Oriented Convexity Number} is $\mathcal{NP}$-complete.

\begin{proof}[of Claim \ref{C1}]
We will consider two cases.   First, assume that $|C \cap V(H_u)| \ge 2$.   Let $w_1$ and $w_2$ be vertices in $C \cap V(H_u)$.   Clearly, $w_1 H_u w_2$ is a $w_1w_2$-geodesic in $D$ and $w_2 H_u w_1$ is a $w_2w_1$-geodesic in $D$.   Hence, $V(H_u) \subseteq C$.

For the second case suppose that $w_1 \in V(H_u)$ and $w_2 \in C \setminus V(H_u)$.   Therefore, every $w_1w_2$-directed path in $D$ uses the vertex $x_u$, and every $w_2w_1$-directed path in $D$ uses the vertex $y_u$.   Thus, $|C \cap V(H_u)| \ge 2$ and we are back to the first case.
\end{proof}

\begin{proof}[of Claim \ref{C2}]
Suppose first that $z_j \in C$ with $i \ne j$.   Assume without loss of generality that $i < j$.   Every $z_j z_i$-directed path in $D$ uses the vertices $z_1$ and $z_4$, thus $z_1, z_4 \in C$.   For every $u \in V(G)$, $(z_4, y_u) \cup (y_u H_u x_u) \cup (x_u, z_1)$ is a $z_4 z_1$-geodesic in $D$.   We conclude from Claim \ref{C1} that  $V(H_u) \subseteq C$ for every $u \in V(G)$, and it follows that $C = V(G)$.

If $z_i \in C$ for $1 \le i \le 4$ and $w \in C$ for some $w \in V(D) \setminus \{z_1, z_2, z_3, z_4\}$, then every $wz_i$-directed path in $D$ uses the vertex $z_1$ and every $z_iw$-directed path in $D$ uses the vertex $z_4$.   Thus, $z_1, z_4 \in C$ and we are done.
\end{proof}

\begin{proof}[of Claim \ref{C3}]
It follows from Claim \ref{C1} that $V(H_u) \cup V(H_v) \subseteq C$.   Since $d_G(u,v) \ge 2$, it is easy to observe that $(x_u, z_1, z_2, z_3, z_4, y_v)$ is an $x_u y_v$-geodesic in $D$.   Hence, $z_1 \in C$ and Claim \ref{C2} guarantees $C=V(D)$.
\end{proof}

\begin{proof}[Proof of Claim \ref{C4}]
It is direct to verify that $V(H_u)$ is a convex set of $D$ for every $u \in V(G)$.   Let $|S| \ge 2$ and $u,v \in S$.   If $w_1 \in V(H_u)$ and $w_2 \in V(H_v)$, then every $w_1 w_2$-directed path in $D$ uses the vertices $x_u$ and $y_v$; moreover, it must contain an $x_uy_v$-directed path.   Our claim follows from noting that $(x_u, y_v) \in A(D)$. 
\end{proof}
\end{proof}

It is easy to observe that the directed $6$-cycles in the previous construction can be replaced by directed $2n$-cycles, and the directed path $(z_1, \dots, z_4)$ can be replaced by a directed path of length $n$, in order to get the result for an oriented bipartite graph of arbitrary large girth.

\section{Grids} \label{SGrids}

We begin this section with two straightforward lemmas regarding convex sets in strong oriented graphs.   The first lemma is a direct observation, so the proof will be omitted.

\begin{lemma} \label{girthcon}
Let $D$ be a strongly connected oriented graph.  If $C \subseteq V(D)$ is a convex set such that $|C| \ge 2$, then $C$ induces a strong subdigraph of $D$.    Therefore, $\min (S_{SC} (G) \setminus \{ 1 \}) \ge g(G)$, where $g(G)$ stands for the girth of $G$.
\end{lemma}

\begin{lemma} \label{comcon}
Let $D$ be an oriented graph.   If $C \subseteq V(D)$ is a maximal convex set, then $D - C$ is a connected subdigraph of $D$.
\end{lemma}

\begin{proof}
Otherwise, let $D_1, \dots, D_n$ be the connected components of $D-C$.   Since $C$ is a convex set of $D$, $C \cup \bigcup_{i=1}^{n-1} V(D_i)$ is a convex set of $D$ properly containing $C$, a contradiction.
\end{proof}

The following observation is simple, but also very useful while searching for adequate orientations to realize specific convex numbers.   The proof is straightforward and thus will be omitted.

\begin{observation} \label{ai-ao}
Let $G$ be a triangle-free graph and  let $D$ be an orientation of $G$ with a convex set $C$.   Let $x \in V \setminus C$ be adjacent to a vertex $y \in C$.
\begin{itemize}
	\item If $(x,y) \in A(D)$, then $N(x) \cap C \subseteq N^+ (x)$.
	\item If $(y,x) \in A(D)$, then $N(x) \cap C \subseteq N^- (x)$.
\end{itemize}
\end{observation}

Let $n$ be an integer.   We denote by $P_n$ the path on $n$ vertices, and we will assume without loss of generality that $P_n = (1, \dots, n)$.   The $(n \times m)$-\emph{grid} is the cartesian product $P_n \Box P_m$.   Hence, if $G = P_n \Box P_m$, then $V(G) = \{ (i,j) \colon\ 1 \le i \le n, 1 \le i \le m \}$.   Although it is not standard, and it can be impractical in a different context, for the sake of simplicity we will denote the ordered pair $(i,j)$ as $i_j$.   Also, we will use the canonical embedding of the grid $G = P_n \Box P_m$ in the plane to define orientations of $G$, and directed paths in an orientation $D$ of $G$.   To achieve this goal, we will use directions to ``move'' on the grid, denoted as a sequence of movements using the symbols $u, d, l, r$, which stand for up, down, left and right.   As an example, consider a directed cycle denoted in the usual way, \emph{i.e.}, $(i_j, i_{j+1}, (i+1)_{j+1}, (i+1)_j, i_j)$; with our notation we have the sequence $(u,r,d,l)$, starting at vertex $i_j$.   In the following paragraph, there is another example of the orientations that can be defined in this way.

Let $n, m \ge 2$ be integers and $G$ be the canonical plane embedding of the grid $P_n \Box P_m$.   Let $H^\ast$ be a connected subgraph of the interior dual of $G$ and let $H$ be the subgraph of $G$ induced by the faces in $V(H^\ast)$ (hence, the interior dual of $H$ is $H^\ast$).   We define a \emph{whirlpool} to be an oriented graph obtained from $H$ by the following orientation of its edges.

\begin{figure}
\begin{center}
\includegraphics{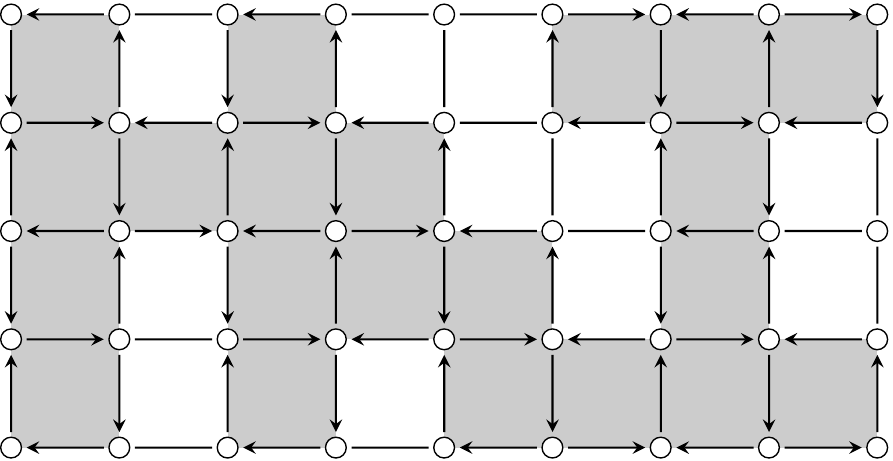}
\caption{A whirlpool orientation of the subgraph induced by the vertices in the gray region.} \label{whirlpool}
\end{center}
\end{figure}

$$\textnormal{Orient } \left\{ \begin{array}{l} i_j i_{j+1} \textnormal{ as:} \left\{ \begin{array}{lc} (i_j,i_{j+1}) & \textnormal{if } i \stackrel{2}{\equiv} j \\ \\ (i_{j+1},i_j) & \textnormal{otherwise.} \end{array}\right.  \\  \\  i_j (i+1)_j \textnormal{ as:} \left\{ \begin{array}{lc} ((i+1)_j,i_j) & \textnormal{if } i \stackrel{2}{\equiv} j \\ \\ (i_j,(i+1)_j) & \textnormal{otherwise.} \end{array} \right. \end{array} \right.$$   An example of a whirlpool is depicted in Figure \ref{whirlpool}, where the graph $H^\ast$ has the gray faces of the grid as its vertex set.   In the rest of the figures, the gray squares will always correspond to directed cycles.   As the following proposition shows, whirlpools have a very important property related to convexity.

\begin{proposition} \label{lcon1grids}
If $D$ is a whirlpool, then $\con (D) = 1$.
\end{proposition}

\begin{proof}
Let $D$ be a whirlpool.   Hence, there exist a pair of integers $n, m \ge 2$ and a subgraph $H$ of the grid $P_n \Box P_m$ such that $D$ is obtained from $H$ by the aforementioned orientation of its edges.   We affirm that every $4$-cycle in $H$ is an oriented $4$-cycle in $D$.   If $C$ is a $4$-cycle in $H$, then $C = (i_j, i_{j+1},(i+1)_{j+1},(i+1)_j,i_j)$ for some $1 \le i \le n-1, 1 \le j \le m-1$.   It is not difficult to observe that if $i \stackrel{2}{\equiv} j$, then $C$ is an oriented cycle in $D$.   Else, $C^{-1}$ (the cycle $C$ in reverse order) is an oriented cycle in $D$.   Hence, $D$ is strongly connected.

Since every edge of $H$ belongs to a $4$-cycle, it follows that every arc of $D$ belongs to a directed $4$-cycle.   Recalling that $D$ is triangle free, it is clear that $u \to v$ implies $d(v,u) = 3$.   Hence, if $u \to v$, then the vertex set of every $4$-cycle containing the arc $(u,v)$ is contained in the convex hull of $\{u,v\}$.   But $D$ is strongly connected, so using the fact that the interior dual of $H$ is connected, it can be shown inductively that for every pair $u,v$ of vertices of $D$, the convex hull of $\{ u, v \}$ is $V(D)$.   Therefore, $\con (D) = 1$.

\end{proof}

\begin{corollary} \label{con=1}
For every grid $G$, $1 \in S_{SG}(G)$.
\end{corollary}

Observe that if an oriented grid $D$ contains a whirlpool $W$ as a subdigraph, then every convex set containing at least two adjacent vertices of $W$ must contain $V(W)$.   Also, if $W$ is a whirlpool, then the digraph $\overleftarrow{W}$ obtained by the reversal of every arc of $W$ has the same properties as $W$; we will call such a digraph an \emph{anti-whirlpool}.

For a given integer $k$ and an oriented graph $D$, it is easier to prove $k \in S_{SC} (D)$ than proving $k \notin S_{SC} (D)$.   The following result excludes some values from the strong convexity spectra of grids.   Although simple, the complete proof of the lemma is to long to be included here.   The proof of the first item of the lemma is complete, as well as the cases $i=3$ and $i=4$ of the second item.   The proofs for the cases $i=5$ and $i=6$ can be obtained with similar arguments.

\begin{lemma} \label{forbid}
Let $n, m \ge 2$ be integers.   If $G = P_n \Box P_m$, then:
\begin{itemize}
	\item $2,3,5, |V|-1 \notin S_{SC} (G)$.
	\item For every $i \in \{ 3, 4, 5, 6 \}$, if $n,m \ge i$, then $|V|-(i-1) \notin S_{SC} (G)$.
\end{itemize}
\end{lemma}

\begin{proof}
By Theorem \ref{no2} we have $2 \notin S_{SC} (G)$.   Observing that every connected subdigraph of $G$ with $3$ or $5$ vertices has at least one vertex of degree $1$, and thus does not admit a strong orientation, it follows from Lemma \ref{girthcon} that $3,5 \notin S_{SC} (G)$.    Any strong orientation of $G$ has neither sinks nor sources.   Also, $g(G)=4$ and hence the orientations of $G$ cannot have transitive vertices.   It follows that $|V|-1 \notin S_{SC} (G)$.

Let $n,m$ be integers such that $n,m \ge 3$ and suppose that a strong orientation $D$ of $G$ has a convex set $C$ of cardinality $|V|-2$.   Let $S=\{ x, y \}$ be the set $V \setminus C$ and assume without loss of generality that $(x,y) \in A(D)$.   Since $D$ is strong, $d^-(x) \ge 1$ and $d^+(y) \ge 1$.   It follows from Observation \ref{ai-ao} that $N^-(x) \cap C = N^-(x)$ and $N^+(y) \cap C = N^+(y)$.   Let us denote by $x^u, x^d, x^l, x^r$ the vertices above, below, to the left and to the right of $x$, respectively, in G.   Since $n,m \ge 3$, we can assume without loss of generality that either $(x^u, x, y, y^r)$ or $(x^l, x, y, y^u)$ is a directed path in $D$, and hence a geodesic.   But this contradicts that $C$ is a convex set.   Hence, a convex set of cardinality $|V|-2$ cannot exist.

Let $n,m$ be integers such that $n,m \ge 4$ and suppose that a strong orientation $D$ of $G$ has a convex set $C$ of cardinality $|V|-3$.   Let $S=\{ x, y, z \}$ be the set $V \setminus C$.   Let us assume without loss of generality that $(x,y) \in A(D)$ and $x^r = y$.   We have two cases.

First consider $y^r = z$.   Again, we have two cases.   Our first subcase is $(z,y) \in A(D)$.   As in the previous argument, either $(x^l, x, y, y^u)$ or $(z^r, z, y, y^u)$ is a geodesic in $D$, contradicting that $C$ is a convex set.   Our second subcase is $(y,z) \in A(D)$.   We will assume without loss of generality that $(y^u, y) \in A(D)$.   Hence, either $(x^l, x, y, z, z^u)$ or $(y^u, y, z, z^r)$ is a geodesic in $D$, contradicting that $C$ is a convex set.

As a second case, consider $y^d = z$, with two subcases.   Our first subcase is $(z,y) \in A(D)$.   Either $(x^d, x, y, y^r)$ or $(z^l, z, y, y^u)$ is a geodesic in $D$, a contradiction.   The second subcase is $(y,z) \in A(D)$.   Hence, at least one of the following directed paths is present in $D$, and it is a geodesic: $(x^d, x, y, y^u), (x^d, x, y, y^r), (y^u, y, z, z^l), (y^r, y, z, z^l)$.   But every case results in a contradiction.   Hence, a convex set of cardnality $|V|-3$ cannot exist.

\end{proof}

\section{Convex spectra of small grids} \label{SCSSG}

The following pair of results deal with the convexity spectra of $n \times 2$ grids.

\begin{lemma} \label{2*n-part1}
Let $n \ge 2$ be an integer and let $G$ be the grid $P_n \Box P_2$.   If $j$ is an integer such that $j \ne 1$ and $\left\lfloor \frac{n}{2} \right\rfloor \le j \le n-1$, then $2j \in S_{SC} (G)$.
\end{lemma}

\begin{proof}
Let $D_1$ be the orientation of $G_1 = P_j \Box P_2$ as a whirlpool.   We will consider two cases.

First, suppose that $j \ge \frac{n}{2}$.   Let $D_2$ be the orientation of $G_2 = G-G_1$ as a whirlpool, if $j$ is odd, or as anti-whirlpool if $j$ is even.   In either case, the orientation of the edges $j_1 j_2$ and $(j+1)_1 (j+1)_2$ result in parallel arcs, \emph{i.e.}, we have either the arcs $(j_1,j_2)$ and $((j+1)_1, (j+1)_2)$ or the arcs $(j_2,j_1)$ and $((j+1)_2,(j+1)_1)$.   We will assume that $j$ is odd, the remaining case can be dealt similarly.   If $D$ is the digraph obtained by orienting the two remaining edges as $(j_1,(j+1)_1)$ and $((j+1)_2,j_2)$, then it is clear that $\partial^+(V(D_1)) = \{ (j_1,(j+1)_1) \}$ and $\partial^-(V(D_1)) = \{ ((j+1)_2,j_2) \}$.   But $d_D(j_1,j_2)=1$, and hence, $V(D_1)$ is a convex set of cardinality $2j$.   If $C$ is a convex set of $P_n \Box P_2$ such that $|C| > 2j$, then $C \cap V(D_1) \ne \varnothing \ne C \cap V(D_2)$.    From the previous observations about $\partial^+ (V(D_1))$ and $\partial^- (V(D_1))$, we conclude that $j_1, j_2, (j+1)_1, (j+1)_2 \in C$.   Since $D_1$ and $D_2$ are whirlpools, we obtain $C=V(P_n \Box P_2)$.   Therefore, $V(D_1)$ is a maximum convex set of $D$.

If $j = \left\lfloor \frac{n}{2} \right\rfloor$, then we will assume that $n$ is odd, otherwise we are back in the previous case.   Let $D_2$ be the orientation of $G_2 = G- G[V(D) \setminus V(G_1) \cup \{ (j+1)_1, (j+1)_2 \}]$ as a whirlpool, if $j$ is odd, or as anti-whirlpool if $j$ is even.   Again, we will assume that $j$ is odd.   Orient the remaining edges in the following way:   $((j+1)_1, (j+1)_2), ((j+1)_2, j_2), ((j+1)_2, (j+2)_2), (j_1, (j+1)_1)$ and $ ((j+2)_1, (j+1)_1)$.   An argument similar to the one used in the previous case shows that $V(G_1)$ and $V(G_2)$ are convex sets of $D$, and clearly $|V(G_1)| = |V(G_2)|$.   If $C$ is a convex set of $P_n \Box P_2$ such that $|C| > 2j$, then $C \cap \{ (j+1)_1, (j+1)_2 \} \ne \varnothing$.   Since $N^+((j+1)_1) = \{ (j+1)_2 \}$ and $N^-((j+1)_2) = \{ (j+1)_1 \}$, then $\{ (j+1)_1, (j+1)_2 \} \subseteq C$.   But $((j+1)_2, j_2, (j-1)_2, (j-1)_1, j_1, (j+1)_1)$ and $((j+1)_2, (j+2)_2, (j+3)_2, (j+3)_1, (j+2)_1, (j+1)_1)$ are $(j+1)_2 (j+1)_1$-geodesics in $D$.   Since $G_1$ and $G_2$ are whirlpools in $D$, we obtain $C = V(D)$.   Hence, $V(G_1)$ is a maximum convex set of $D$.

\end{proof}

\begin{theorem}
If $n \ge 2$ is an integer and $G$ is the grid $P_n \Box P_2$, then $$S_{SC} (G) = \{ 1 \} \cup \{ 2j \colon\ \left\lfloor \tfrac{n}{2} \right\rfloor \le j \le n-1\} \setminus \{ 2 \}.$$
\end{theorem}

\begin{proof}
Let $D$ be a strong orientation of $G$ and $C$ a maximum convex set of $D$.   It follows from Lemma \ref{comcon} and the fact that $D[C]$ is strong that, for some $2 \le j \le n-1$, either $V(C) = V(P_j \Box P_2)$, or $V(C) = V(D) \setminus V(P_j \Box P_2)$.   Hence, there are not odd integers greater than $1$ in $S_{SC} (G)$.

Assume without loss of generality that $C = V(P_j \Box P_2)$ for some $2 \le j \le n-1$.   Also assume without loss of generality that $(j_1, j_2) \in A(D)$.   We will consider two cases.

Consider for the first case $((j+1)_2, (j+1)_1) \in A(D)$.   Since $D$ is strong, either $((j+1)_1, j_1), (j_2, (j+1)_2) \in A(D)$ or $(j_1, (j+1)_1), ((j+1)_2, j_2) \in A(D)$.   In the former case $(j_2, (j+1)_2, (j+1)_1, j_1)$ is a $j_2 j_1$-geodesic in $D$, a contradiction.   In the latter case, it is direct to verify that $V(D) \setminus C$ is also a convex set; since $C$ is maximum, we get $j \ge \frac{n}{2}$.

For the second case consider $((j+1)_1, (j+1)_2) \in A(D)$.   Again, since $D$ is strong, either $((j+1)_1, j_1), (j_2, (j+1)_2) \in A(D)$ or $(j_1, (j+1)_1), ((j+1)_2, j_2) \in A(D)$.   In the former case it is easy to check that $V(D) \setminus C$ is a convex set of $D$, hence, $j \ge \frac{n}{2}$.   In the latter case, if $((j+2)_2, (j+1)_2), ((j+1)_1, (j+2)_1) \in A(D)$, then $V(P_{j+1} \Box P_2)$ is a convex set of $D$, a contradiction.   Hence $((j+1)_2, (j+2)_2), ((j+2)_1, (j+1)_1) \in A(D)$ and we have two further cases.

If $((j+2)_2, (j+2)_1) \in A(D)$, then $V(D) \setminus C$ is a convex set of $D$, thus $j \ge \frac{n}{2}$.

If $((j+2)_1, (j+2)_2) \in A(D)$, then $V(D) \setminus (C \cup \{ (j+1)_1, (j+1)_2 \})$ is a convex set of $D$, and hence $j \ge \left\lfloor \frac{n}{2} \right\rfloor$.

\end{proof}

Although a bit more complex, similar arguments can be used for the $n \times 3$ grids.

\begin{lemma} \label{3*n-c=3k}
Let $n \ge 2$ be an integer and let $G$ be the grid $P_n \Box P_3$.   If $j$ is an integer such that $2 \le j \le n-1$, then $3j \in S_{SC} (G)$.
\end{lemma}

\begin{figure}
\begin{center}
\includegraphics[width=\textwidth]{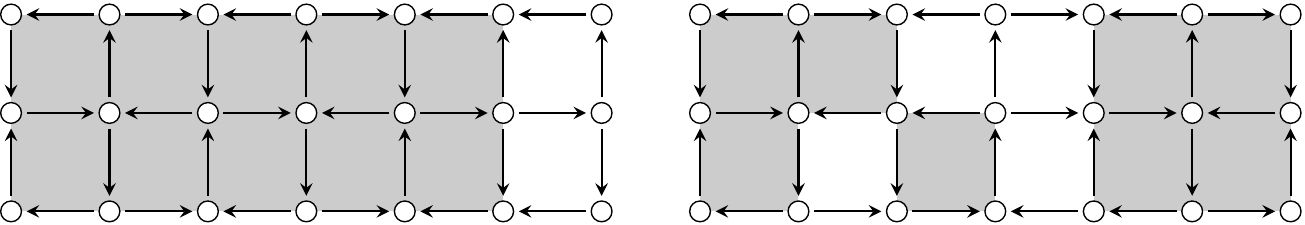}
\caption{The orientations used in the proof of Lemma \ref{3*n-c=3k} with $n=7$ and $j \in \{ 4, 6 \}$.} \label{3*n-c=3kFig}
\end{center}
\end{figure}

\begin{proof}
We consider two cases.   If $j=n-1$, then let $D_1$ be the orientation of $G_1 = P_{n-1}  \Box P_3$ as a whirlpool.   Suppose that $n$ is even, the remaining case is analogous.   Orient the remaining edges as $((n-1)_3, n_3), ((n-1)_1, n_1), (n_2, (n-1)_2), (n_1, n_2), (n_3, n_2)$.   It is straightforward to verify that this orientation $D$ of $G$ is strong and $V(D_1)$ is a maximum convex set of $D$.

For $2 \le j \le n-2$, let $D_1$ be the orientation of $G_1 = P_{j-2} \Box P_3$ as a whirlpool (note that $D_1$ is empty for $j=2$).   Let $D_2$ be the orientation of $G_2 = G - (P_j \Box P_3)$ as a whirlpool.   Orient $(j_3, (j-1)_3, (j-1)_2, (j-1)_1, j_1, j_2, j_3)$ as a directed cycle.   Also, orient $\partial (G_1)$ as $((j-2)_3, (j-1)_3), ((j-1)_2, (j-2)_2), ((j-2)_1, (j-1)_1)$ if $j$ is even, and reverse each of these arcs if $j$ is odd.   Finally, orient the remaining arcs as $(j_2, (j-1)_2), (j_3, (j+1)_3), (j_2, (j+1)_2), ((j+1)_1, j_1)$.   Let $D$ be the resulting orientation of $G$, and let $D_3$ be the induced subdigraph $D[ V(P_j \Box P_3)]$ of $D$.

We will assume that $j$ is even, the remaining case can be dealt similarly.   Observe that $\partial^+(D_3) = \partial^-(D_2) = \{ (j_3, (j+1)_3), (j_2, (j+1)_2) \}$, and $\partial^-(D_3) = \partial^+(D_2) = \{ ((j+1)_1, j_1) \}$.   Since $d((j+1)_2, (j+1)_1)=3$, $d((j+1)_3, (j+1)_1)=4$, $d(j_3, j_1)=4$, and $d(j_2, j_1)=3$, it is clear that $V(D_3)$ is a convex set of $D$ with $|V(D_3)|=3j$.

Let $C$ be a convex set of $D$ such that $|C| > 3j$.   Since every convex set in a strong digraph induces a strong subdigraph, we observe that $|C \cap V(D_2)| \ge 2$.   But $D_2$ is a whirlpool, and hence $V(D_2) \subseteq C$.   Note that $d((j+1)_1, (j+1)_3) = 4$ and also that $((j+1)_1, j_1, j_2, j_3, (j+1)_3)$ is a directed path in $D$.   Thus, $j_1, j_2, j_3 \in C$.   Recall that $d(j_3, j_1)=4$ and consider the directed path $(j_3, (j-1)_3, (j-1)_2, (j-1)_1, j_1)$ to conclude $(j-1)_1, (j-1)_2, (j-1)_3 \in C$.   Since $j$ is even, $((j-1)_3, (j-1)_2, (j-2)_2, (j-2)_3, (j-1)_3)$ is a directed cycle in $D$.   This implies $|C \cap V(D_1)| \ge 2$, but $D_1$ is a whirlpool, and hence $V(D_1) \subseteq C$ and $C=V(D)$.   Therefore $\con (D) = 3j$.

\end{proof}

\begin{lemma} \label{3*n-c=3k+2}
Let $n \ge 2$ be an integer and let $G$ be the grid $P_n \Box P_3$.   If $j$ is an integer such that $2 \le j \le n-2$, then $3j+2 \in S_{SC} (G)$.
\end{lemma}

\begin{figure}
\begin{center}
\includegraphics[width=\textwidth]{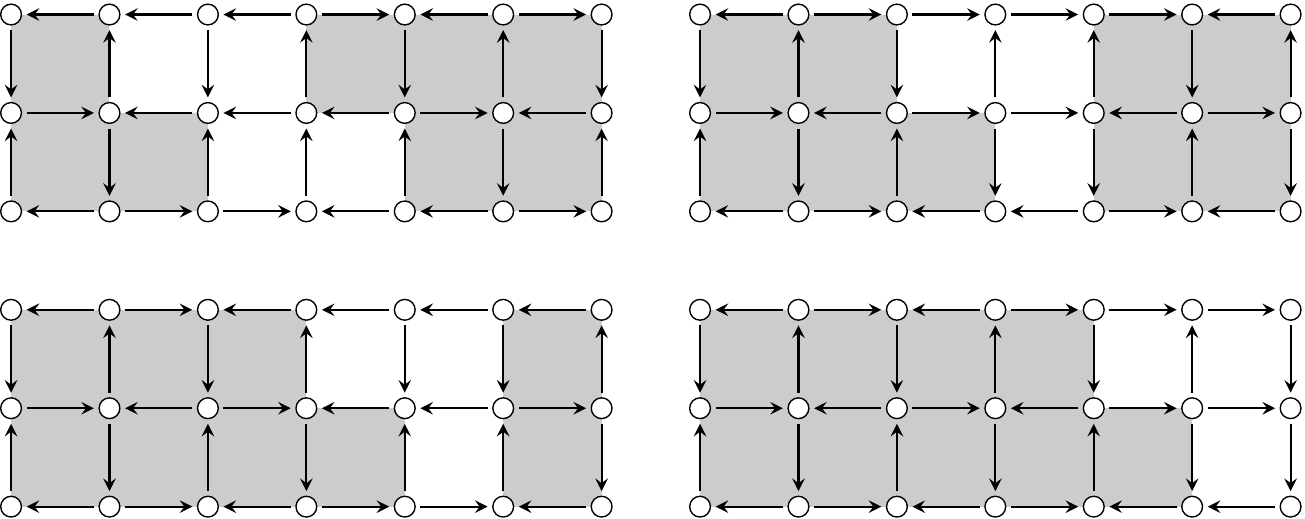}
\caption{The orientations used in the proof of Lemma \ref{3*n-c=3k+2} with $n=7$ and $j \in \{2,3,4,5\}$.} \label{3*n-c=3k+2Fig}
\end{center}
\end{figure}

\begin{proof}
Let $D_1$ be the orientation of $G_1 = (P_{j+1} \Box P_3) - (j+1)_3$ as a whirlpool.   We will consider two cases.

For the first case consider $\left\lfloor \frac{n}{2} \right\rfloor \le j \le n-3$, and let $D_2$ be the orientation of $G_2 = G - (P_{j+1} \Box P_3)$ as a whirlpool, if $j$ is even, or as an anti-whirlpool, if $j$ is odd.

If $j$ is even, orient the remaining edges as $((j+1)_3, j_3), ((j+1)_3, (j+1)_2), ((j+2)_3, (j+1)_3), ((j+2)_2, (j+1)_2), ((j+1)_1, (j+2)_1)$ to obtain the orientation $D$ of $G$.   It is direct to verify that $D$ is strong.   Observing that $d((j+1)_1, (j+1)_2) = 1$ and $d((j+1)_1, j_3) = 3$ it is easy to verify that $V(D_1)$ is a convex set of $D$.   If $C$ is a convex set of $D$ such that $|C| > 3j+2$, then $|C \cap (V(D) \setminus V(D_1))| \ne \varnothing$.   But $\partial^+(D_1) = \{ ((j+1)_1, (j+2)_1) \}$, and hence $|C \cap V(D_2)| \ge 2$.   Recalling that $D_2$ is a whirlpool and observing that $((j+2)_3, (j+1)_3, j_3)$ is a $(j+2)_3 j_3$-geodesic in $D$, we conclude that $C = V(D)$.   Hence $\con (D) = 3j+2$.

If $j$ is odd, orient the remaining edges as $(j_3, (j+1)_3), ((j+1)_2, (j+1)_3), ((j+1)_3, (j+2)_3), ((j+1)_2, (j+2)_2), ((j+2)_1, (j+1)_1)$ to obtain the orientation $D$ of $G$.   This orientation is, locally, the dual orientation of the case when $j$ is even, so analogous arguments show that $V(D_1)$ is a convex set and $\con (D) = 3j+2$.

As a second case, assume that $2 \le j < \left\lfloor \frac{n}{2} \right\rfloor $ or $j = n-2$.   When $j$ is odd, orient $(j_3, (j+1)_3, (j+2)_3, (j+2)_2, (j+2)_1, (j+1)_1)$ as a directed path, and orient the arcs $((j+1)_2, (j+1)_3), ((j+1)_2, (j+2)_2)$.   If $j \ne n-2$, let $D_2$ be the orientation of $G_2 = G \setminus (P_{j+2} \Box P_3)$ as an anti-whirlpool and orient the remaining edges of $G$ as $((j+2)_1, (j+3)_1), ((j+2)_2, (j+3)_2), ((j+3)_3, (j+2)_3)$ to obtain $D$.   Clearly $D$ is strong.   Also, it is direct to verify that $V(D_1)$ is a convex set of $D$ with $3j+2$ vertices.

Let $C$ be a convex set of $D$ such that $|C| > 3j+2$.   If $|C \cap V(D_2)| \ge 2$, then $V(D_2) \subseteq C$.   But $((j+3)_3, (j+2)_3, (j+2)_2, (j+2)_1, (j+3)_1)$ is a $(j+3)_3 (j+3)_1$-geodesic in $D$, and thus, $(j+2)_i \in C$ for $1 \le i \le 3$.   Also, $((j+2)_1, (j+1)_1, j_1, j_2, (j+1)_2, (j+1)_3, (j+2)_3)$ is a $(j+2)_1 (j+2)_3$-geodesic in $D$.   This implies $C=V(D)$, because $|C \cap V(D_1)| \ge 2$.

Otherwise, and because $j \ge 2$, $V(D_1) \subseteq C$ and $v \in C$ for some $v \in \{ (j+1)_3, (j+2)_1, (j+2)_2, (j+2)_3 \}$.   In any case, $(j_3, (j+1)_3, (j+2)_3, (j+2)_2, (j+2)_1, (j+1)_1)$ is the union of a $j_3 v$-geodesic and a $v (j+1)_1$-geodesic in $D$ (and the case $j=n-2$ is finished) .   Since $((j+2)_2, (j+3)_2, (j+3)_3, (j+2)_3)$ is a $(j+2)_2 (j+2)_3$-geodesic in $D$, we have $|C \cap V(D_2)| \ge 2$ and $C=V(D)$.

When $j$ is even, as in the previous case, we can orient the remaining edges of $G$ to obtain, locally, an orientation that is dual to the orientation when $j$ is odd.   Hence, analogous arguments can be followed to prove that $V(D_1)$ is a maximum convex set of $D$.

Therefore, $\con(D)=3j+2$.

\end{proof}

\begin{lemma} \label{3*n-c=3k+1}
Let $n \ge 2$ be an integer and let $G$ be the grid $P_n \Box P_3$.   If $j$ is an integer such that $3 \le j \le n-2$, then $3j+1 \in S_{SC} (G)$.
\end{lemma}

\begin{figure}
\begin{center}
\includegraphics[width=\textwidth]{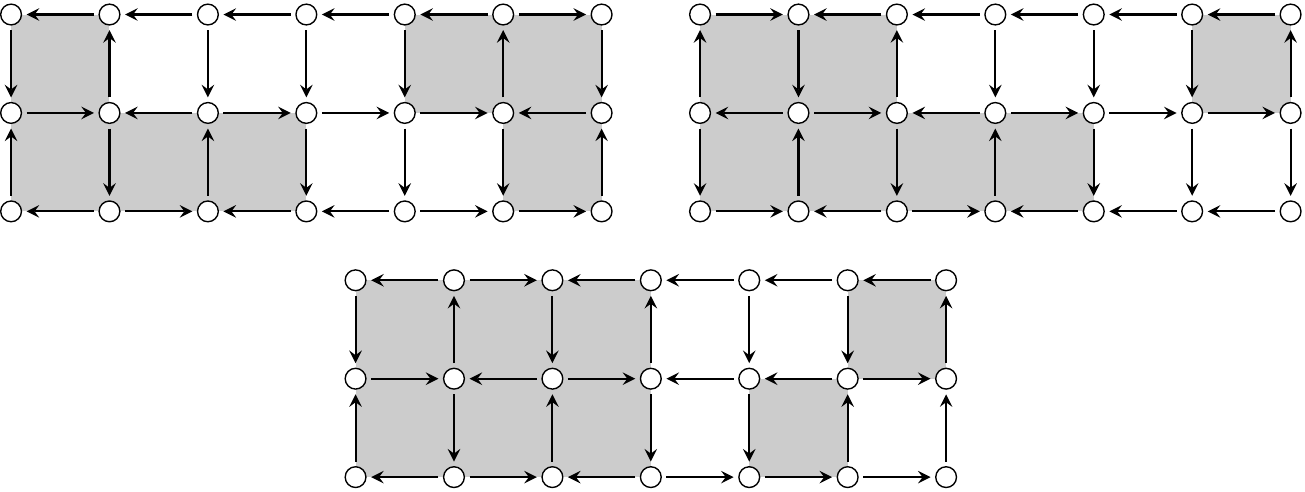}
\caption{The orientation used in the proof of Lemma \ref{3*n-c=3k+1} for $n=7$ and $j \in \{3, 4, 5 \}$.} \label{3*n-c=3k+1Fig}
\end{center}
\end{figure}

\begin{proof}
We will assume that $j$ is odd, the remaining case can be dealt similarly.   Orient $G'_1 = P_{j-1} \Box P_3$ as a whirlpool to obtain $D'_1$.   We will consider two cases.

For the first case, suppose that $j = n-2$.   Let $D_1$ be the digraph obtained from $G_1 = (P_{j+1} \Box P_3) - \{ j_3, (j+1)_3 \}$ by orienting $G'_1$ as $D'_1$, $((j-1)_1, j_1, (j+1)_1, (j+1)_2, j_2, (j-1)_2)$ as a directed path, and the remaining edge of $G[G_1]$ as $(j_2, j_1)$.   Also, orient $((j+1)_1, (j+2)_1, (j+2)_2, (j+2)_3, (j+1)_3, j_3, (j-1)_3)$,  and $((j+1)_3, (j+1)_2, (j+2)_2)$ as directed paths.   Orient the remaining edge as $(j_3, j_2)$ to obtain the digraph $D$.   It is immediate to verify that $D$ is a strong digraph.

Note that $\partial^+ (D_1) = \{ ((j+1)_1, (j+2)_1), ((j+1)_2, (j+2)_2) \}$ and $\partial^- (D_1) = \{ (j_3, (j-1)_3), (j_3, j_2), ((j+1)_3, (j+1)_2) \}$.   Observing that $d((j+1)_2, j_2) = d((j+1)_1, (j+1)_2) = 1$, $d((j+1)_1, j_2)=2$, $d((j+1)_1, (j-1)_3)=4$, and $d((j+1)_2, (j-1)_3)=3$, it is easy to conclude that $V(D_1)$ is a convex set of $D$ with $3j+2$ vertices.

Let $C$ be a convex set of $D$ such that $|C| > 3j+1$.   Since $j=n-2$, then $|C \cap V(D'_1)| \ge 2$ and hence $V(D'_1) \subseteq C$.   Also, there is at least one vertex $v \in C \cap (V(D) \setminus V(D_1))$.   Regardless of the choice of $v$, the directed path starting at $(j-1)_1$ and defined by the sequence $(r,r,r,u,u,l,l,l)$ results from the union of a $(j-1)_1 v$-geodesic and a $v (j-1)_3$-geodesic.   Hence, $V(D) \subseteq C$ and therefore $\con (D) = 3j+1$

As a second case, assume that $j \le n-3$.   Let $D_1$ and $D_2$ be the digraphs obtained by orienting both $G_1 = (P_{j+1} \Box P_3) - \{ j_3, (j+1)_3 \}$ and $G_2 = G - (V(G_1) \cup \{ j_3, (j+1)_3, (j+2)_1 \})$ as whirlpools.   Orient $((j+1)_3, (j+1)_2, (j+2)_2, (j+2)_1, (j+1)_1)$ and $((j+2)_3, (j+1)_3, j_3,(j-1)_3)$ as directed paths.   If $j = n-3$, orient $((j+3)_1, (j+2)_1)$, and orient the same edge as $(j+2)_1, (j+3)_1$ otherwise.   Finally, orient the remaining edges as $(j_3, j_2)$ to obtain the digraph $D$.   It is direct to verify that $D$ is strong.

Observe that $\partial^+ (D_1) = \{ ((j+1)_2, (j+2)_2) \}$ and $\partial^- (D_1) = \{ (j_3, j_2), (j_3, (j-1)_3), ((j+1)_3, (j+1)_2), ((j+2)_1, (j+1)_1) \}$.   Noting that $d((j+1)_2, (j+1)_1)=1$, $d((j+1)_2, j_2)=3$ and $d((j+1)_2, (j-1)_3) = 5$, it is not hard to verify that $V(D_1)$ is a convex set of $D$.

Let $C$ be a convex set of $D$ such that $|C| > 3j+1$.   If $j=n-3$, then $|V(D_1) \cap C| \ge 2$, and $V(D_1) \subseteq C$.    If $j < n-3$ and  $|C \cap V(D_2)| \ge 2$, then $V(D_2) \subseteq C$.    But $((j+2)_2, (j+2)_1, (j+3)_1)$ is a $(j+2)_2 (j+3)_1$-geodesic in $D$, which implies $(j+2)_1 \in C$.   The directed path with initial vertex $(j+2)_1$ and defined by the sequence $(l,l,u,r,r)$ is a $(j+2)_1 (j+2)_2$-geodesic in $D$.   From here we observe that $|V(D_1) \cap C| \ge 2$ and thus $V(D_1) \subseteq C$.   If $j < n-3$ and $|C \cap V(D_2)| \le 1$, then $|V(D_1) \cap C| \ge 2$, and $V(D_1) \subseteq C$.   Hence, in every case $V(D_1) \subseteq C$.   Since there is at least one vertex from $V(D) \setminus V(D_1)$ in $C$, necessarily $(j+2)_2 \in C$.   The directed path starting at $(j+2)_2$ and defined by the sequence $(r,u,l,l,l,l)$ is a $(j+2)_2 (j-1)_3$-geodesic, and hence $V(D_2) \subseteq C$.   But if $V(D_2) \cup V(D_1) \subseteq C$, it is easy to verify that $C = V(D)$.   Hence, $\con (D) = 3j+1$.

\end{proof}

So far, we have every integer of the convexity spectrum of $P_n \Box P_3$, except for $4$.   Our next theorem deals with the remaining case.

\begin{lemma} \label{n*3-c=4}
If $n \ge 3$ is an integer and $G$ is the grid $P_n \Box P_3$, then $4 \in S_{SC} (G)$
\end{lemma}

\begin{figure}
\begin{center}
\includegraphics{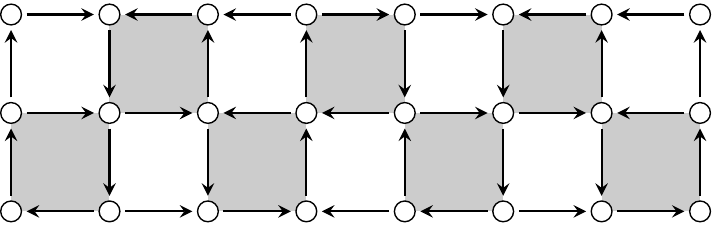}
\caption{The orientation used in the proof of Lemma \ref{n*3-c=4} for $n=8$.} \label{n*3-c=4Fig}
\end{center}
\end{figure}

\begin{proof}
Consider the standard plane embedding of $G$ and color the interior faces gray and white with a checkerboard-like pattern, coloring the square on the bottom left corner with gray.   We will define an orientation of the arcs of $G$ using this coloring, an example can be seen on Figure \ref{n*3-c=4Fig}.

There are two rows of squares.   Enumerate the gray squares in each row from left to right.   Orient the bottom left corner square as a whirlpool and, from here, orient all the gray squares in its row alternating whirlpool and anti-whirlpool orientations.   Orient all the gray squares in the upper row following the same principle, but start orienting as an anti-whirlpool the first gray square.   At this point, every arc dividing two interior faces of $G$ has received an orientation.   Every remaining unoriented edge $e$ of $G$ divides a white square from the exterior face of $G$.   Thus, the edge $e$ lies in exactly one square of $G$, and has one parallel arc $a$ in the square.   It $e$ is not an edge of a corner square, orient it in the same direction as $a$.   There are four edges belonging to the white corner squares that remain unoriented.   Orient the remaining edges as $2$-paths in such way that there are not white oriented squares.   Let $D$ be the digraph obtained by this orientation.   Clearly, $D$ is strong and the vertices of each gray square conform a convex set.   

Let $C$ be a convex set of $D$ such that $|C| > 4$.   There must be two gray squares $S_1$ and $S_2$ such that $v \in V(S_1) \cap V(S_2)$ and $V(S_1) \cup V(S_2) \subseteq C$.    Since $v$ is an interior vertex, it belongs to two white squares.   Let $u$ be a vertex in the opposite corner in one of these white squares $S_3$.   Assume without loss of generality that $v$ is the lower left corner of $S_1$, the upper right corner of $S_2$, and the lower right corner of $S_3$; the remaining cases can be dealt similarly.

If $u$ is the middle vertex of a $S_1 S_2$-path or a $S_2 S_1$-path of length two, then $u \in C$, and hence $V(S_3) \subseteq C$.

Otherwise, let $x$ and $y$ be the upper right and lower left corners of $S_3$, respectively.   Hence, $d(x,y) = d(y,x)=4$ and either the sequence $(l,l,d,r)$ starting from $x$ determines an $xy$-directed path of length $4$, or the sequence $(l,u,r,r)$ determines a $yx$-directed path of length $4$.   In either case, $V(S_3) \subseteq C$.

It can be verified inductively that $C = V(D)$, and hence, $\con (D) = 4$. 

\end{proof}

\begin{theorem}
If $n \ge 3$ is an integer and $G$ is the grid $P_n \Box P_3$, then $$S_{SC} (G) = [1,3n-3]  \setminus \{ 2, 3, 5, 7\}.$$
\end{theorem}

\begin{proof}
By virtue of Theorem \ref{no2} and Lemmas \ref{forbid}, \ref{3*n-c=3k}, \ref{3*n-c=3k+2}, \ref{3*n-c=3k+1} and \ref{n*3-c=4}, it remains to prove that $7 \notin S_{SC} (G)$.   Let $D$ be a strong orientation of $G$ and $C$ a convex set of $D$.   Lemmas \ref{girthcon} and \ref{comcon} imply that $C$ induces a strong subdigraph of $D$ and that $V(D) \setminus C$ induces a connected subdigraph of $D$, respectively.   But every connected subgraph of $G$ with $7$ vertices has either a vertex of degree $1$, and thus does not admit a strong orientation; or does not have a connected complement.   Hence, $7 \notin S_{SC} (G)$.

\end{proof}

\section{Convex spectra of general grids} \label{SMain}

The following lemma is the cornerstone of the vast majority of the arguments we will use in this section.

\begin{figure}
\begin{center}
\includegraphics{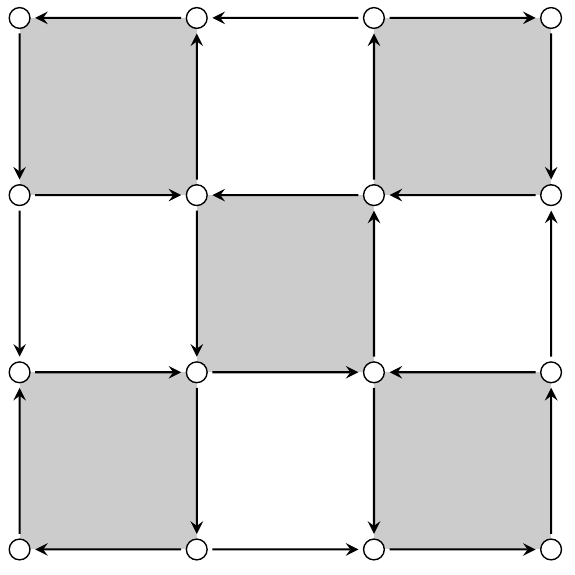}
\caption{The digraph $H$: an orientation of $P_4 \Box P_4$ with convexity number $4$.} \label{4*4-c=4}
\end{center}
\end{figure}

\begin{lemma}
The oriented graph $H$ in Figure \ref{4*4-c=4} has convexity number $4$.
\end{lemma}

\begin{proof}
It is easy to check that the vertices on the boundary of each of the gray filled squares conform a convex set.   We affirm that any convex set in $H$ has at most $4$ vertices.   Let $C$ be a maximum convex set in $H$ and suppose that $|C| > 4$.   Observe that the intersection of $C$ with the vertices of each gray filled square is either empty, or it has one vertex, or it has four vertices.   Hence, since $|C| >  4$, the vertices of at least two squares are contained in $C$.   If the vertices of two gray filled squares different from the central one are contained in $C$, then there are at least two vertices of the central square in $C$; thus, the vertices of the central square are contained in $C$.   So, by symmetry, we need only to consider two cases.

The first case is when the square on the lower left corner and the central square are contained in $C$.   Assume that the vertex in the lower left corner of $H$ is $(1,1)$.  It is easy to check that $((2,3), (2,4), (1,4), (1,3), (1,2))$ and $((2,1), (3,1), (4,1), (4,2), (3,2))$ are $(2,3)(1,2)$- and $(2,1)(3,2)$-geodesics, respectively.   From here, $(4,2), (3,3) \in C$, and $((4,2), (4,3), (3,3))$ is a $(4,2)(3,3)$-geodesic in $H$.   Therefore, there are at least two vertices of each gray filled square in $C$ and we can conclude that $C = V(H)$, a contradiction.

In the second case, we have the upper left corner and the central square contained in $C$.   Now, $((1,3), (1,2), (2,2))$ is a $(1,3)(2,2)$-geodesic in $H$.   Hence, there are at least two vertices of the lower left corner square in $C$.   So, the lower left corner square is contained in $C$ and we have the condition of the first case.

Since contradictions are obtained in both cases, we conclude that $|C| \le 4$.   Hence, $\con(H) = 4$.

\end{proof}

As the reader would expect, the main part of the argument in the next lemma's proof is the construction of the orientation.   The following lemmas will use similar orientations, so the descriptions will be very detailed in the first ones, and will loose detail as the lemmas progress.

\begin{lemma} \label{n*m-c=4}
Let $n,m \ge 4$ be integers.   If $G= P_n \Box P_m$ is a grid, then $4 \in S_{SG} (G)$.
\end{lemma}

\begin{proof}
Consider the standard plane embedding of $G$ and color the interior faces gray and white with a checkerboard-like pattern, assigning gray to the square on the bottom left corner (like in Figure \ref{4*4-c=4}).   We will define an orientation of the arcs of $G$ using this coloring.

Enumerate the rows of squares from bottom to top.   Enumerate the gray squares in each row from left to right.   Orient the bottom left corner square as a whirlpool and, from here, orient all the gray squares in the first column and first row alternating whirlpool and anti-whirlpool orientations.   Now, the first square or every odd row is oriented, so we can orient all the gray squares in the odd rows alternating whirlpool and anti-whirlpool orientations.   A similar idea can be used to orient all the gray squares in even rows, but start orienting as an anti-whirlpool the first gray square on the second row.   At this point, every arc dividing two interior faces of $G$ has received an orientation.   We will consider two cases.

First, suppose that $n$ and $m$ are even integers, hence every corner square of $G$ is gray, and every remaining unoriented edge $e$ of $G$ divides a white square from the exterior face of $G$.   Thus, the edge $e$ lies in exactly one square of $G$, and has one parallel arc $a$ in the square.   Orient $e$ in the same direction as $a$.   All the edges of $G$ are now oriented; let $D$ be the resulting oriented graph.   Figure \ref{4*4-c=4} is an example of this orientation.   It is easy to verify that $D$ is strongly connected, and the vertices of every gray square conform a convex set of $D$.   If $C$ is a convex set of $D$ such that $|C| > 4$, then $C$ intersects the vertices on at least two different gray squares $S_1$ and $S_2$ in odd columns and rows.   Since $C$ induces a strong subdigraph of $D$, we may assume without loss of generality that $S_1$ and $S_2$ are gray squares in the same row and adjacent odd columns.   If $S_1$ and $S_2$ are in row $i$ and columns $j$ and $j+2$, then the vertices of $S_1$ and $S_2$, together with the vertices of the squares $S_3$ and $S_4$ in row $i+2$ and columns $j$ and $j+2$, induce a subdigraph of $D$ isomorphic to the digraph $H$ of Figure \ref{4*4-c=4}.   Therefore $\bigcup_{i=1}^4 V(S_i) \subseteq C$.   We can repeat this argument using squares $S_2$ and $S_4$ and the squares in column $j+4$ and rows $i$ and $i+2$.   Iterating this process we obtain $C = V(D)$.   Hence, $\con(D)=4$.

For the second case, assume that $n$ or $m$ is an odd integer.   By virtue of Lemma \ref{n*3-c=4}, we assume that $n, m \ge 4$.   Observe that there are exactly two white corner squares in $G$.   Except for the edges in the white corner squares, orient the remaining edges of $G$ as in the previous case.   For each of the white corner squares we have the two cases depicted in Figure \ref{n*m-c=4Fig} (the squares we are interested in are the bottom right corners), and two isomorphic cases obtained by reversing all the arcs of the previous ones.   We will consider two cases.

\begin{figure}
\begin{center}
\includegraphics{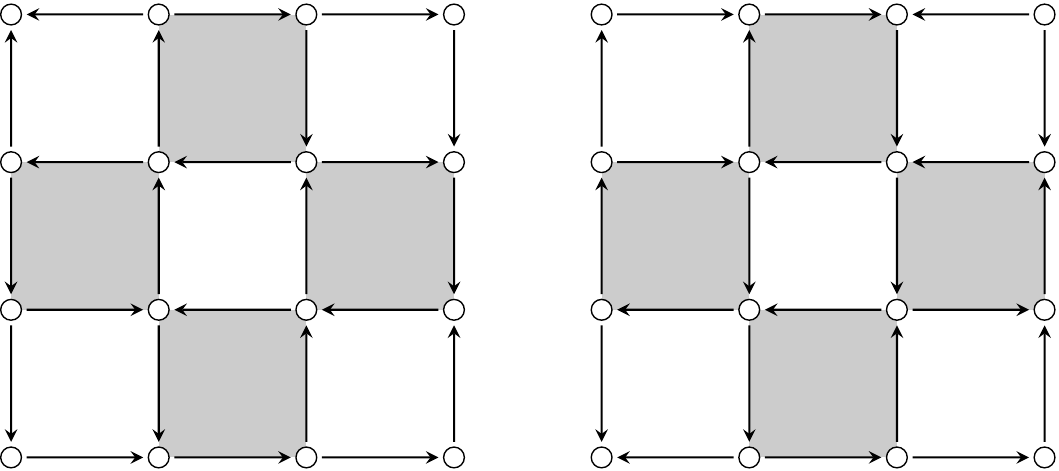}
\caption{The two non-isomorphic cases for the white corners in the proof of Lemma \ref{n*m-c=4}.} \label{n*m-c=4Fig}
\end{center}
\end{figure}

If $n \ne 4 \ne m$, complete the orientation $D$ of $G$ as in Figure \ref{n*m-c=4Fig}.   Again, it is direct to verify that $D$ is strong and the vertices of every gray square conform a convex set of $D$.   Let $C$ be a convex set of $D$ such that $|C| > 4$.   As in the previous case, it suffices to show that there are two consecutive gray squares in the same row or the same column that intersect $C$.   Since $|C| > 4$, there $C$ intersects at least two different gray squares.   Hence, the desired condition is held, unless the gray squares are precisely those adjacent to one of the white corners.   We will assume without loss of generality that one white corner square is the one in the bottom right, as in Figure \ref{n*m-c=4Fig}.  Let $S_1$ and $S_2$ be gray squares to the left and above the white corner, respectively.   Let  $v_i$ be the vertex in the upper left corner of $S_i$, $i \in \{ 1, 2 \}$.

Consider first the situation depicted by the digraph on the left in Figure \ref{n*m-c=4Fig}.   It is clear that $d(v_1, v_2) = 4$, and also that, starting from $v_1$, the sequence $(u,u,r,d)$ determines a $v_1 v_2$-directed path of length $4$.   Hence, $C$ intersects two consecutive gray squares in the same row.

For the situation depicted by the digraph on the right in Figure \ref{n*m-c=4Fig}, it is clear that $d(v_1, v_2)=6$.   It is also clear that, starting from $v_1$, the sequence $(l,u,r,u,r,d)$ determines a $v_1 v_2$-directed path of length $6$.   Since $C$ is convex, it intersects two consecutive gray squares in the same row.

As a final case, assume without loss of generality that $m=4$.   Since $n$ is an odd integer, the two white corners are those on the right side of $G$.   In the situation depicted by the digraph on the left in Figure \ref{n*m-c=4Fig}, use precisely that orientation and the same argument as in the previous case.   In the remaining case, use the orientation of Figure \ref{n*m-c=4Fig} for the bottom right corner, and orient the upper right corner also as a directed path of length $2$ (assume that it goes up and left).   This orientation is strong, the vertices of each gray square conform a convex set, and the same argument as the previous case shows that we can find two consecutive gray squares in the same column that intersect $C$, and hence $C = V(D)$.

\end{proof}

Our first modification to the previous orientation will be getting gray rectangles instead of $2 \times 2$ squares.

\begin{lemma} \label{n*m-c=ab}
Let $n,m \ge 4$ be integers and let $G= P_n \Box P_m$ be a grid.  If $a,b \ge 2$ is a pair of integers such that $a \le n-1$ and $b \le m-1$, then $ab \in S_{SG} (G)$.
\end{lemma}

\begin{proof}
The idea of this proof is to generalize the orientation used in Lemma \ref{n*m-c=4}, but using grids of size $ab$ instead of squares in the odd-numbered rows and columns.   An example of this orientation is depicted in Figure \ref{n*m-c=abFig}.

First, suppose that $a < n-1$ and $b < m-1$.   Enumerate the rows and columns of squares of $G$ from down to up and from left to right, respectively.   Color with gray and white the squares of $G$ in a checkerboard-like pattern, but considering rectangles of squares instead of single squares, in the following way.   Color the squares in the first $a-1$ columns and $b-1$ rows, and the square in the $a$-th column and $b$-th row wih gray.   Color the squares in the $a$-th column and the first $b-1$ rows, and the squares in the $b$-th column and the first $a-1$ rows in white.   Color the rest of $G$ with a tiling of this coloring.

Enumerate the rows and columns of gray rectangles in the pattern from left to right and from down to up.   Clearly, the gray rectangles in the even numbered rows (columns) are squares.     We can assume that $ab > 4$, hence, the gray rectangles in the odd numbered rows (columns) are proper rectangles.

Orient the gray squares in the even rows as in the proof of Lemma \ref{n*m-c=4} (in particular, the first gray square of the second row is an anti-whirlpool).   Orient the gray rectangles in the odd rows as whirlpools or anti-whirpools in such way that every gray square (possibly except the last gray square in every row or every column) in a even row, together with the four gray squares sharing a vertex with it, induce a digraph isomorphic to $H$ (Figure \ref{4*4-c=4}).   Of course, to achieve this end, we also need to orient four additional arcs dividing either two white squares of $G$, or a white square and the exterior face of $G$.   Now, every unoriented edge of $G$ divides two white squares, or a white square from the exterior face in an even row or column of gray squares.   Orient every edge dividing two white squares in the same direction as the closest arc in a gray square in the same row or column.   There are unoriented edges parallel to the arcs oriented in the previous step; orient those edges in the same direction as the arcs they are parallel to.   The remaining unoriented edges form paths on the exterior face of $G$ joining pairs of gray rectangles.    Orient those paths as directed paths in such way that, if any, the corner vertices of $G$ in a white square have in-degree and out-degree equal to one.   Figure \ref{n*m-c=abFig} shows an example of this orientation.

\begin{figure}
\begin{center}
\includegraphics{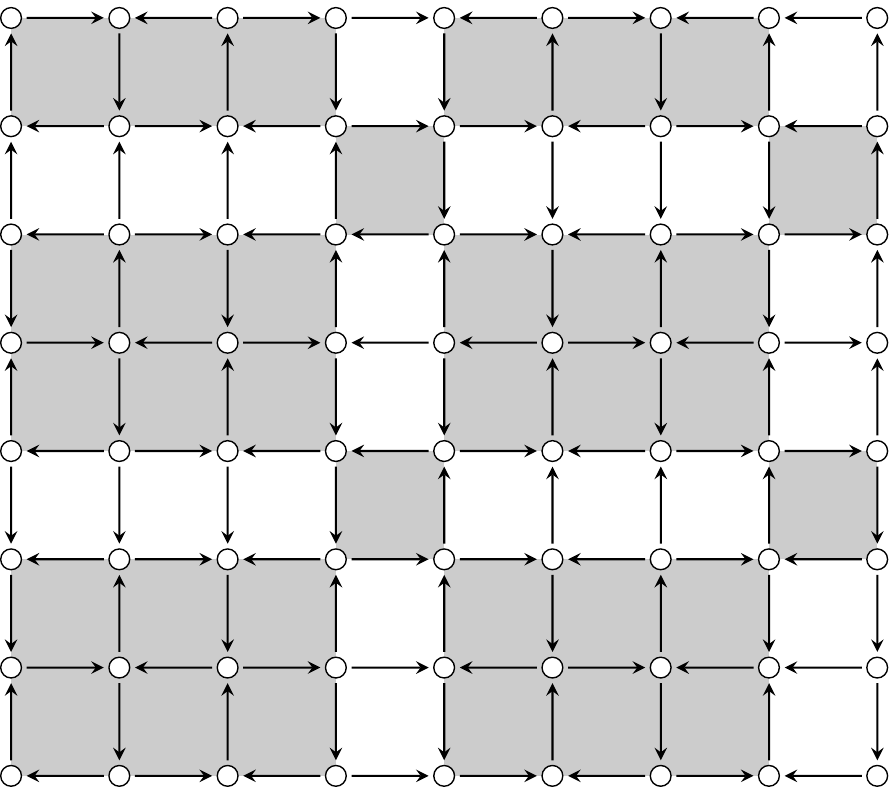}
\caption{The orientation described in the proof of Lemma \ref{n*m-c=ab} for $n=9$, $m=8$, $a=4$ and $b=3$.} \label{n*m-c=abFig}
\end{center}
\end{figure}

If $D$ is the digraph obtained from $G$ by means of the previously described orientation, then, by mimicking the arguments in the proof of Lemma \ref{n*m-c=4} we reach the desired conclusion.

\begin{figure}
\begin{center}
\includegraphics[width=\textwidth]{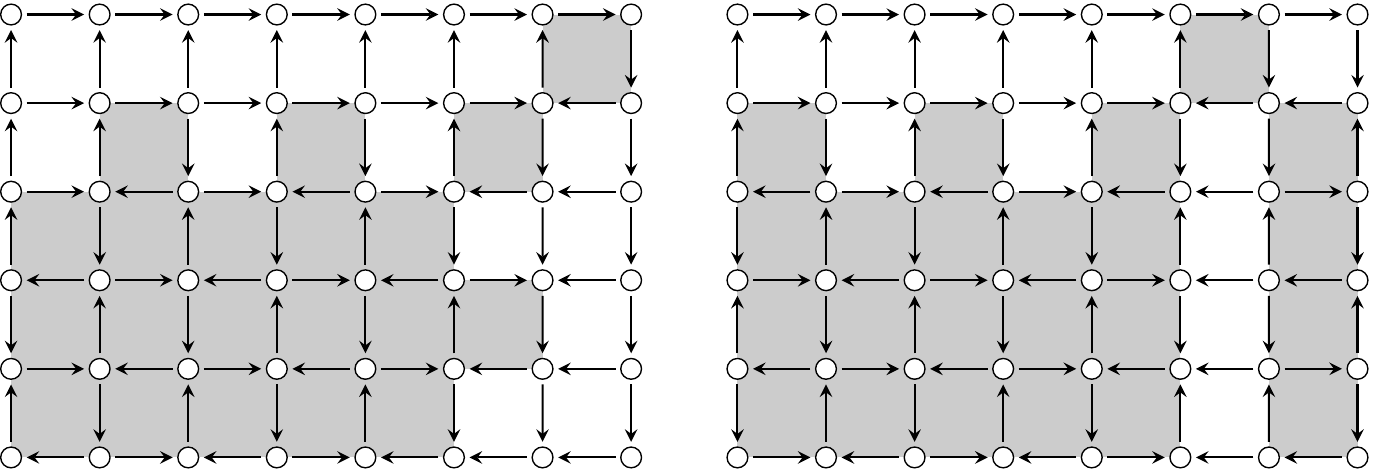}
\caption{The orientation used in the proof of Lemma \ref{n*m-c=ab} for $n=8$, $m=6$, $b=5$ and $a \in \{ 5,6 \}$. } \label{n*m-c=abFig2}
\end{center}
\end{figure}

If $a=n-1$ or $b=m-1$, we have to be careful with the white corners, but the simple modification shown in Figure \ref{n*m-c=abFig2} suffices to use the same argument as in the previous case.

\end{proof}

In the next lemma we use a simpler orientation, which is depicted in Figure \ref{n*m-c=nbFig}.   This orientation resembles the one used in Lemma \ref{2*n-part1}.

\begin{figure}
\begin{center}
\includegraphics[width=\textwidth]{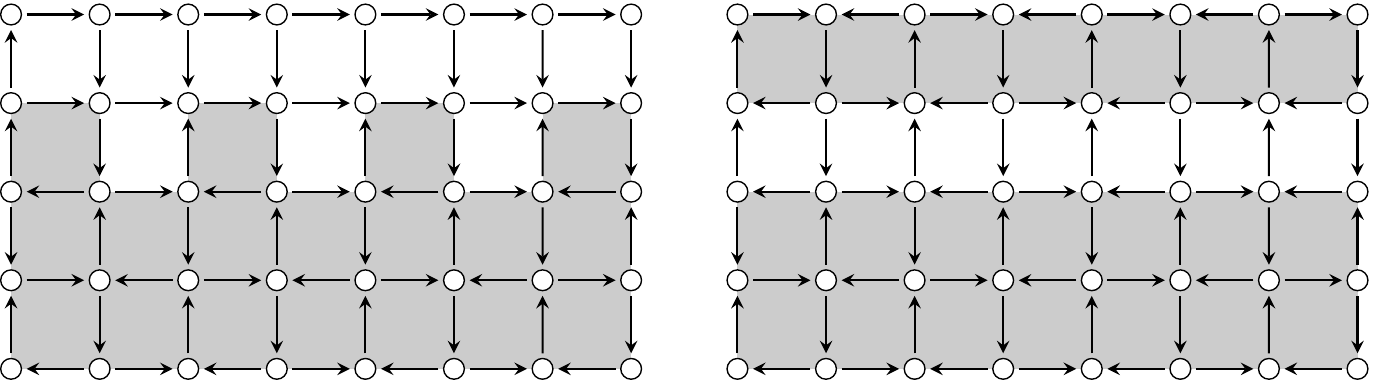}
\caption{The orientations used in the proof of Lemma \ref{n*m-c=nb} for $b=m-1$ (left) and $b<m-1$ (right).} \label{n*m-c=nbFig}
\end{center}
\end{figure}

\begin{lemma} \label{n*m-c=nb}
Let $n,m \ge 4$ be integers and let $G= P_n \Box P_m$ be a grid.  If $b$ is an integer such that $\frac{m}{2} \le b \le m-1$, then $nb \in S_{SG} (G)$.
\end{lemma}

\begin{proof}
First, suppose that $b = m-1$.  Orient $P_n \Box P_b$ as a whirlpool.   Now, re-orient the corresponding arcs in the fiber $P_n^{m-1}$ to obtain a directed path in the same direction (left or right) as the arc between $1_{m-1}$ and $2_{m-1}$.   Also, orient the fiber $P_n^m$ as a directed path in the same direction as the arc between $1_{m-1}$ and $2_{m-1}$.   Assume without loss of generality that $P_n^m$ is oriented right.   Finally orient down every edge in $\partial (P_n^m)$ except for $(1_{m-1}, 1_m)$, which is oriented up.   If $D$ is the resulting oriented graph, then it is easy to verify that $V(P_n \Box P_b)$ is a convex set of $D$, and $\con (D) = nb$.

If $b < m-1$, then orient $G_1 = P_n \Box P_b$ as a whirlpool and $G_2 = G - G_1$ as a whirlpool or as an anti-whirlpool in such way that the fibers $P_n^b$ and $P_n^{b+1}$ are isomorphic.   Orient the remaining edges arbitrarily, as long as there is one arc going up and one arc going down, to obtain $D$.    We will show that $V(G_1)$ is a convex set of $D$.

Let $P$ be a $i_b j_b$-path in $D$ such that every intermediate vertex of $P$ belongs to $V(G_2)$.   Then $P$ can be codified as $(u,x_2, \dots, x_{k-1}, d)$.   Let $(y_2, \dots, y_{k-1})$ be the sequence such that $$y_i = \left\{ \begin{array}{ll} x_i & \textnormal{if } x_i \in \{ l, r \} \\  u & \textnormal{if } x_i = d \\ d & \textnormal{if } x_i = u. \end{array} \right.$$

From the way we oriented $G_1$ and $G_2$ we conclude that $P'=(y_2, \dots, y_{k-1})$ defines a $i_b j_b$-path in $D$, strictly shorter than $P$ and such that $V(P') \subseteq V(G_1)$.   Hence, $V(G_1)$ is a convex set of $D$.

Recalling that $\frac{m}{2} \le b$, and using the fact that every convex set with more than $nb$ vertices has at least two vertices in $V(G_1)$ and at least two vertices in $V(G_2)$, we conclude that $\con (D) = nb$.

\end{proof}

The following three lemmas deal with the most complex cases.

\begin{lemma} \label{n*m-c=ab-k-l}
Let $n,m \ge 4$ be integers and let $G= P_n \Box P_m$ be a grid.  If $a,b,k,l$ are integers such that $0 \le k \le a-2 \le n-3$,   $0 \le l \le b-2 \le m-3$, and $a,b \ge 3$, then $ab-(k+l) \in S_{SG} (G)$.
\end{lemma}

\begin{figure}
\begin{center}
\includegraphics[width=\textwidth]{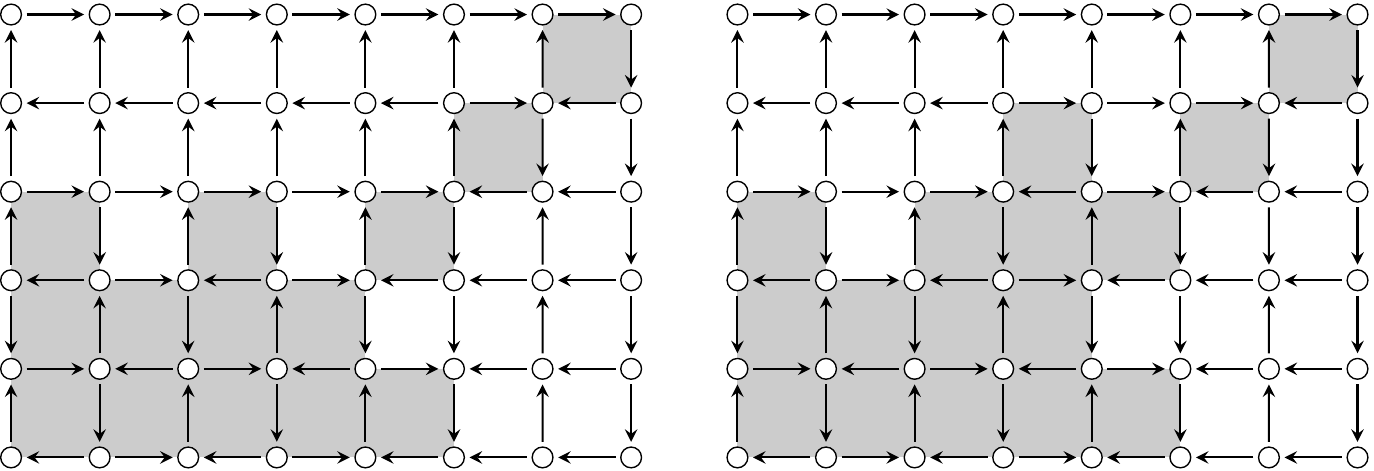}
\caption{The orientation used in the proof of Lemma \ref{n*m-c=ab-k-l} for $n=8$, $m=6$, $a=7$ $b=5$ and $(k,l) \in \{ (5,3), (3,2)  \}$.} \label{n*m-c=ab-k-lFig}
\end{center}
\end{figure}

\begin{proof}
Consider a coloring of the squares of $G$ similar to the coloring used in the proof of Lemma \ref{n*m-c=ab}, with the following differences.   In the lower left corner of the aforementioned orientation of $G$ we have a rectangle, $R_1$, of $(a-1)\times(b-1)$ gray squares.   Consider the subgraph $G_1$ of $G$ obtained by deleting the first $(a-1)$ columns and the first $(b-1)$ rows of vertices of $G$.   Color $G_1$ as in the proof of Lemma \ref{n*m-c=4}.   Finally, complete the checkerboard-like pattern with gray rectangles of size $1 \times (a-1)$ squares in the first row, and with gray rectangles of size $(b-1) \times 1$ squares in the first column.

\begin{figure}
\begin{center}
\includegraphics{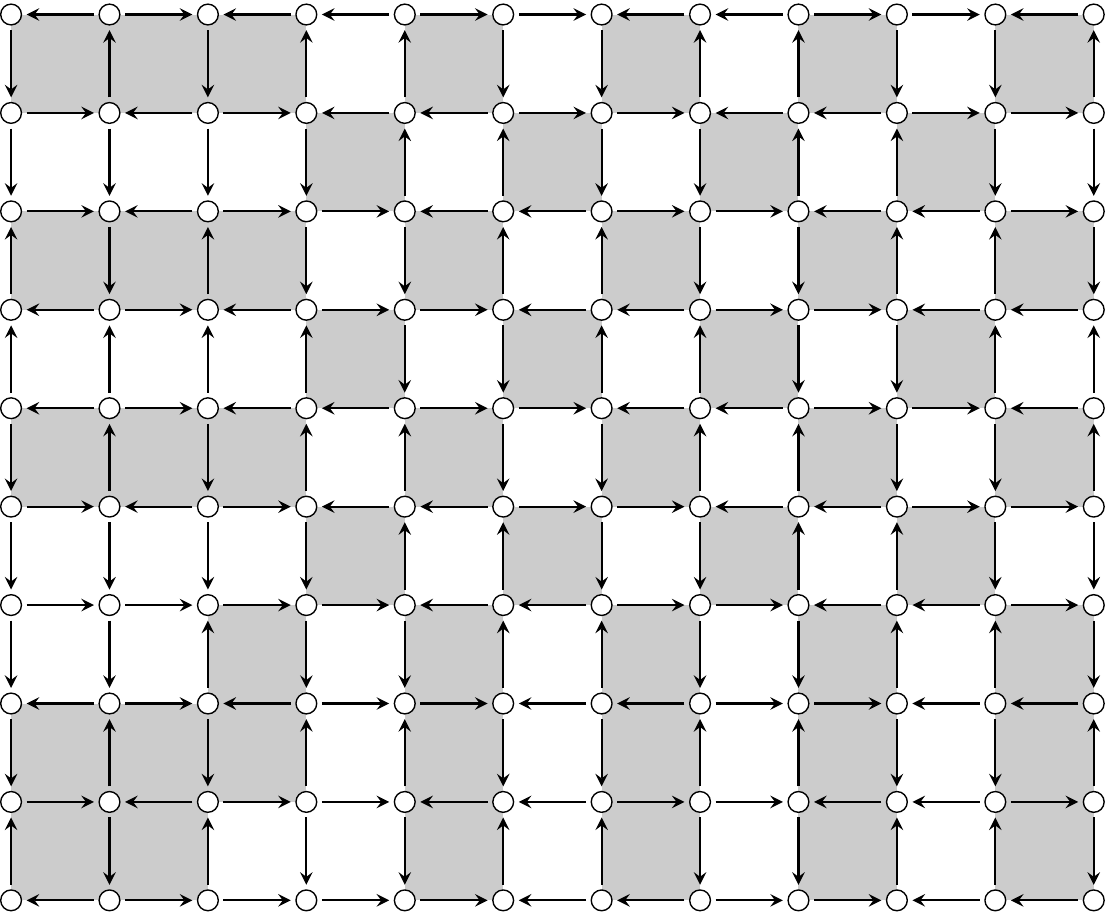}
\caption{The orientation described in the proof of Lemma \ref{n*m-c=ab-k-l} for $n=12$, $m=10$, $a=4=b$, $k=2$ and $l=1$.} \label{n*m-c=ab-k-lFig}
\end{center}
\end{figure}

We want our largest convex set to be a segment of $R_1$, the gray rectangle in the lower left corner.   Color white the squares in the first $k$ columns from the top row of $R_1$.   Also color white the squares in the first $l$ rows of the rightmost column of $R_1$ to obtain the gray region $R$.   Now, orient $G$ as in Lemma \ref{n*m-c=ab}, orienting $R$ as a whirlpool.    Orient all the remaining edges to the right and down to obtain $D$.   An example of this orientation is depicted in Figure \ref{n*m-c=ab-k-lFig}.

Now, observe that the top right corner of $R$ coincides with the top right corner of $R_1$.   Hence, this orientation has the same properties as the orientation of Lemma \ref{n*m-c=ab}.   We can assume without loss of generality that $k+l < \min \{ a, b\}$.   Otherwise, assuming that $b \le a$, we have that $ab-(k+l)= a(b-1)-(k'+l')$ for some pair of integerers $k', l'$ such that $0 \le k' \le k$, $0 \le l' \le l$.   Hence, $ab-(k+l) \ge 2a, 2b$.   From here, it is easy to observe that the only proper convex sets of $D$ are the vertices in each of the gray regions of $D$.   Since the largest one is $R$, which has $ab-(k+l)$ vertices, we conclude $S_{SC} (D) = ab-(k+l)$.

Again, we have to be careful when $a=n-1$ or $b=m-1$.   The corresponding modifications to the previous orientation are depicted in Figure \ref{n*m-c=ab-k-lFig}.

\end{proof}

\begin{lemma} \label{5*4-c=nm-k}
Let $4 \le n,m \le 5$ be integers and let  $G= P_n \Box P_m$ be a grid.  If $k$ is an integer such that $6 \le k \le 2\max \{ n, m \}-1$, then $nm-k \in S_{SG} (G)$.
\end{lemma}

\begin{proof}
First, consider the case $n=5$ and $m=4$.   If $k \in \{ 8,9 \}$, then the result follows from Lemma \ref{n*m-c=ab-k-l}.   The orientations for $k \in \{ 6, 7 \}$ are shown in Figure \ref{5*4-c=nm-kFig}.   It can be easily verified that the set of white vertices is a largest convex set for this orientation.

Now suppose that $n=4=m$.   The case $k=7$ follows from Lemma \ref{n*m-c=ab-k-l}.   The orientation for $k=6$ is depicted in Figure \ref{5*4-c=nm-kFig}.   Again, it is not hard to verify that the sets of white vertices are largest convex sets for each of the orientations.

Finally, suppose that $n=5=m$.   If $k=9$, then the result follows from Lemma \ref{n*m-c=ab-k-l}.   The orientations for $k \in \{6, 7, 8\}$ are depicted in Figure \ref{5*4-c=nm-kFig}.   As in the previous cases, the sets of white vertices are largest convex sets for the corresponding orientation.

\begin{figure}
\begin{center}
\includegraphics[width=\textwidth]{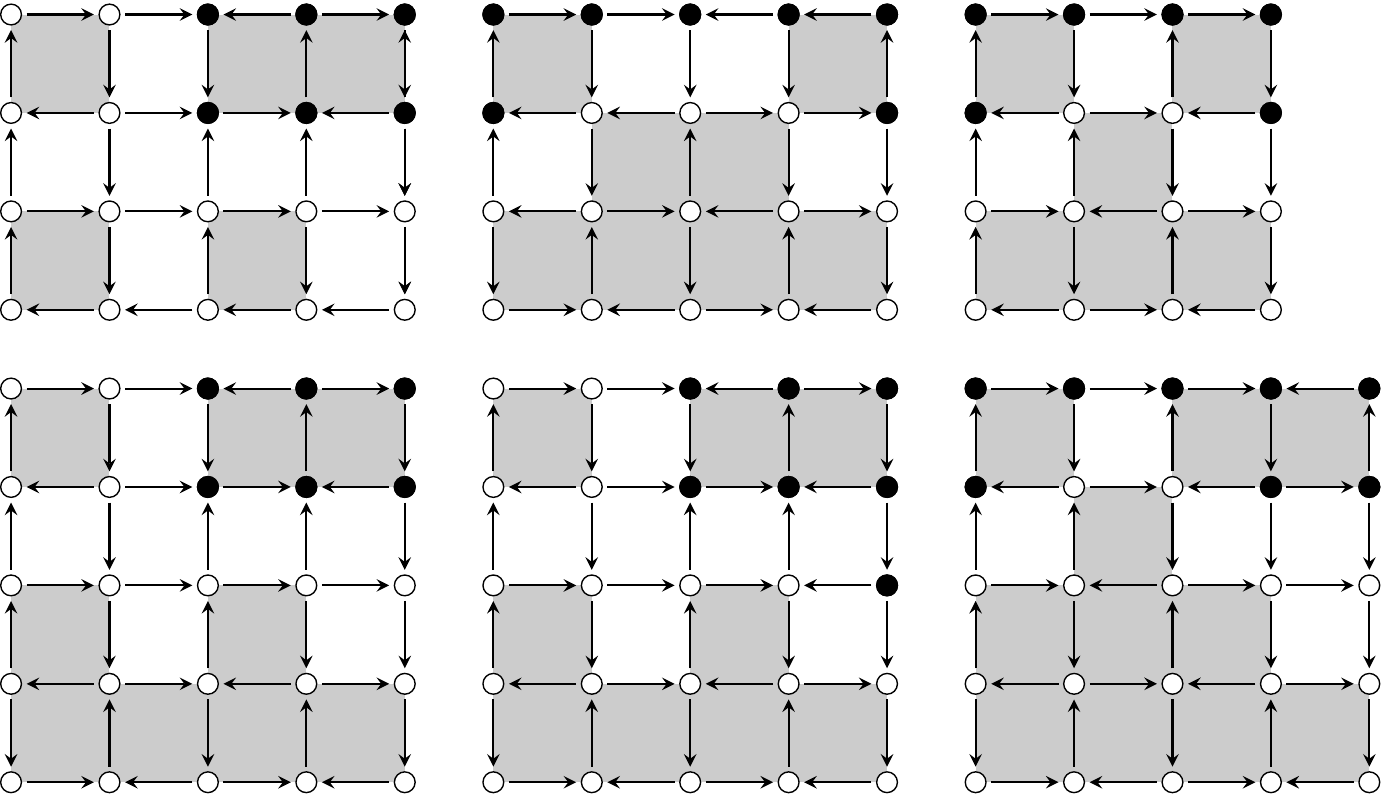}
\caption{The orientations used in the proof of Lemma \ref{5*4-c=nm-k}.} \label{5*4-c=nm-kFig}
\end{center}
\end{figure}

\end{proof}

\begin{lemma} \label{n*m-c=nm-k}
Let $n,m \ge 4$ be integers and let $G= P_n \Box P_m$ be a grid.  If $k$ is an integer such that $6 \le k \le 2\max \{ n, m \}-1$, then $nm-k \in S_{SG} (G)$.   If $m = 4$, then also $nm-5 \in S_{SC} (G)$.
\end{lemma}

\begin{proof}
If $n,m \le 5$, then the result follows from the previous lemma.   Hence, we will suppose without loss of generality that $m \le n$ and $6 \le n$.   Let us consider first that $m \ge 5$.

Observe the orientations of the $5 \times 4$ and $5 \times 5$ grids in Figure \ref{5*4-c=nm-kFig} for $k=6$.   Clearly, the orientation of the latter can be obtained from the orientation of the former by adding an additional row of squares at the bottom of the grid, and orienting this new row as a whirlpool.   Naturally, three arcs of the former orientation should change its direction, in this case, these are the arcs $(1_2, 1_3), (3_1,2_1)$ and $(5_1,4_1)$ of the former oriented graph.   We have a similar situation with the orientation of the $4 \times 4$ and $5 \times 5$ grids for $k=6$ and $k=8$, respectively.   Again, the latter can be obtained from the former by adding one row and one column of squares, orienting the new row as an anti-whirlpool and the new column following certain pattern.   The idea of this proof is to observe that the three orientations of the $5 \times 5$ grids, shown in Figure \ref{5*4-c=nm-kFig}, can be naturally extended to obtain the desired orientations.

\begin{figure}
\begin{center}
\includegraphics[width=\textwidth]{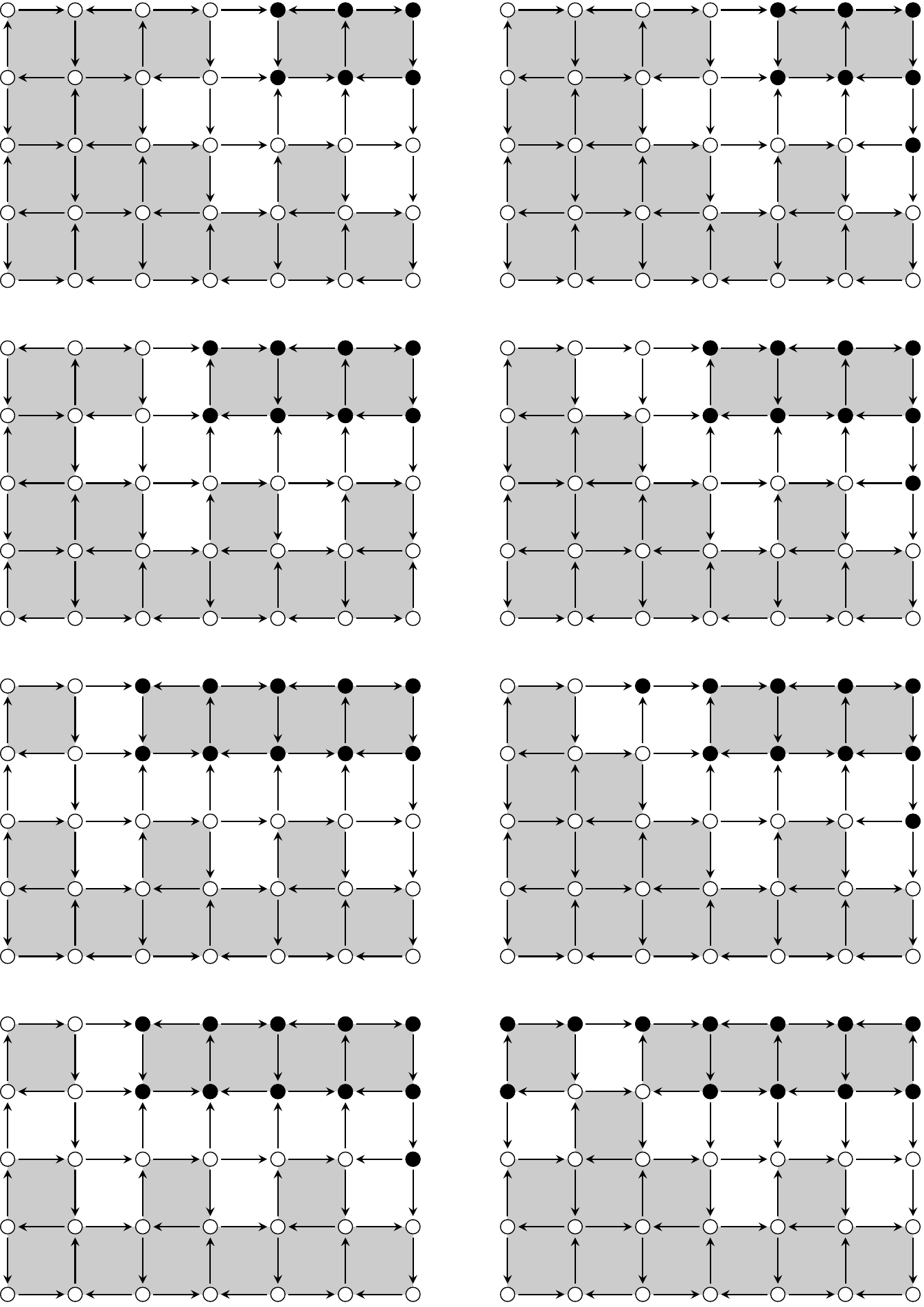}
\caption{The orientations used in the proof of Lemma \ref{n*m-c=nm-k} for $n=7$ and $m=5$.} \label{n*m-c=nm-kFig}
\end{center}
\end{figure}

First, observe that, independently of the value of $n$, we can always add a new row of squares at the bottom of the grid.   Notice that we are extending the largest convex set of the grid, but the complement of such set remains unchanged.   Hence, if we consider the aforementioned orientations of the $5 \times 5$ grid, we can conclude that $6, 7, 8 \in S_{SC} (P_5 \Box P_m)$ for every integer $m \ge 5$.

To add a new column of squares we have two different behaviors.   First, consider the orientations for $k=6$ and $k=7$.   In this case, to add a new column, we need to change the direction of one arc, as in the $5 \times 4$ and $5 \times 5$ grids for the $k=6$ argument.   But, to add further columns we will not need to change the direction of any arc, just add a column of squares oriented as a whirlpool.   Following this procedure we will obtain an orientation of the grid $P_n \Box P_m$ for every pair of integers $n, m \ge 5$, for $k=6$ and $k=7$, respectively.   In the other hand, when $k=8$, adding a new column will not preserve the complement of our largest convex set, the complement will grow larger.   If we add a new column, following the pattern as in the example in Figure \ref{n*m-c=nm-kFig}, we will obtain an orientation of the grid $P_n \Box P_m$ for every pair of integers $n,m \ge 5$, for $k=2n-2$.

We can generalize the orientations used for $k=6$ and $k=7$, respectively, to obtain orientations of the grid $P_n \Box P_m$ for $k = 2n-4$ and $2n-3$.   The idea is to extend the pattern horizontally, we will explain how it is done for the case $k=2n-4$, the remaining case is analogous.   We want to preserve the $2_m n_{m-2}$ directed path defined by the sequence $(d,d,r,r, \dots, r)$, and the whirlpool (or anti-whirlpool) of $n-2$ squares in the upper right corner of the grid.   The square of the upper left corner of the grid is oriented as a whirlpool and the rest of the grid is oriented as a whirlpool or anti-whirlpool, except for some alternating squares adjacent to the aforementioned directed path.   The remaining unoriented edges are oriented as $(1_{m-2}, 1_{m-1})$ and, if it is still unoriented, $(n_{m-2}, n_{m-3})$.   Following a similar idea, orientations for every grid $P_n \Box P_m$ for any pair of integers $n \ge m$ and any $k \in \{ 6, 7 \dots, 2n-2 \}$ can be obtained.   We depict the orientations for $P_7 \Box P_5$ for $k \in \{ 6, 7, 8, 9, 10, 11, 12 \}$ in Figure \ref{n*m-c=nm-kFig}.   Since all these orientations follow a similar pattern, it can be easily proved that each one has convexity number $nm-k$, as desired.

Finally, since we are assuming $n \ge m$, we have $2n-1 \ge n+m-1$, and the case $k \ge n+m-1$ is covered by Lemma \ref{n*m-c=ab-k-l}.

\begin{figure}
\begin{center}
\includegraphics[width=\textwidth]{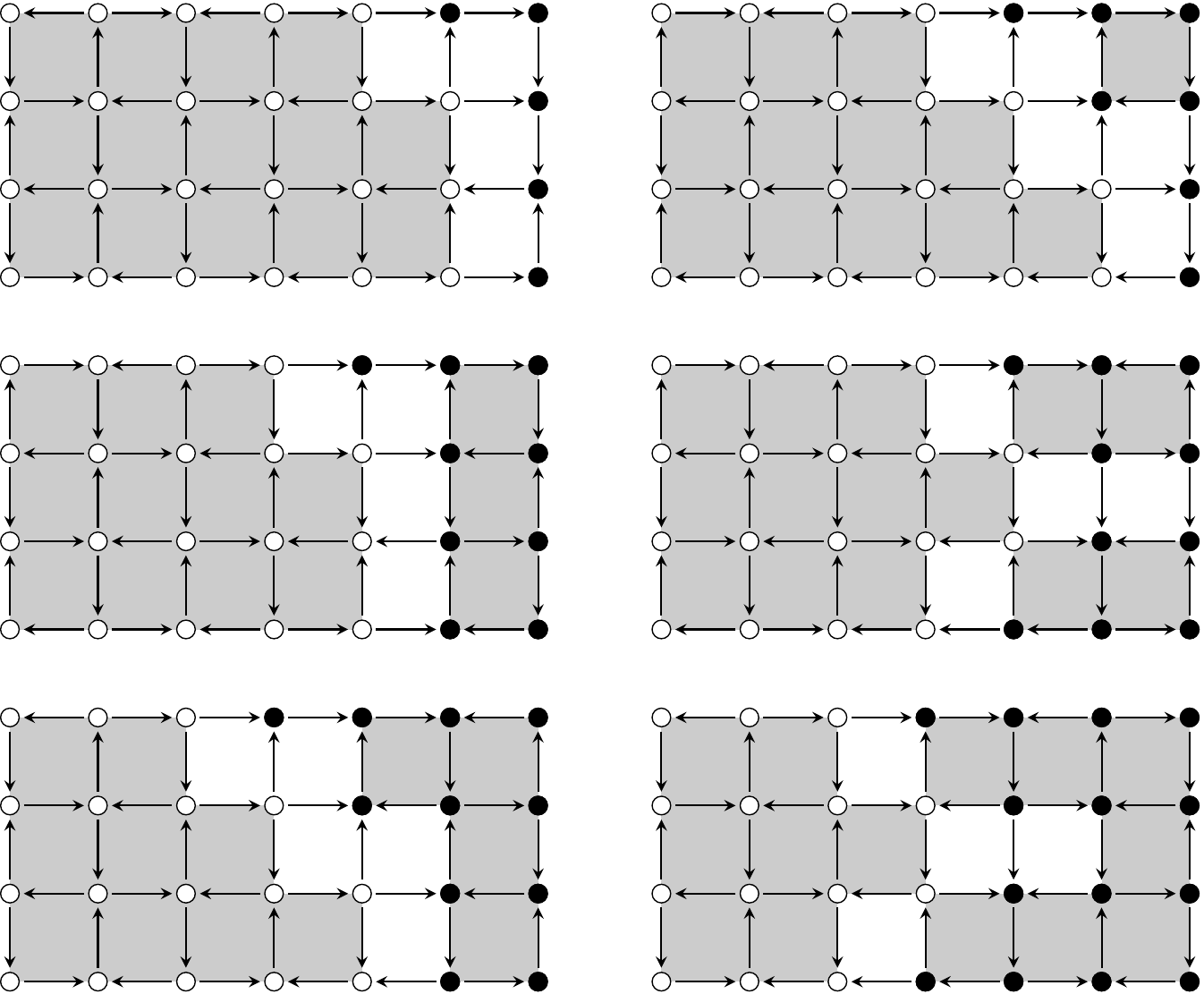}
\caption{The orientations used in the proof of Lemma \ref{n*m-c=nm-k} for $m=4$.} \label{n*m-c=nm-kFig2}
\end{center}
\end{figure}

Now, consider the case $m=4$.   If $k \equiv 0$ (mod 4), then Lemma \ref{n*m-c=nb} gives the desired orientation.   Else, we give basic orientations that can be enlarged adding columns of whirlpools, either to the largest convex set, or to its complement.   The orientation for $k=6$ is the orientation used for the $4 \times 4$ grid (Figure \ref{5*4-c=nm-kFig}).   Clearly, additional rows of whirlpools can be added in the bottom of the grid.   We give the orientations for $k \in \{ 5, 7, 9, 10, 11, 14\}$ in Figure \ref{n*m-c=nm-kFig2}.   Again, the set of black vertices is the complement of the largest set of the oriented graph.   Since it is easy to observe that the given orientations have the desired properties, and when we extend them the arguments remain the same, this concludes the proof ot the lemma.
\end{proof}

We present now the main theorem of this work.

\begin{theorem}
Let $4 \le m \le n$ be a pair of integers and let $G = P_n \Box P_m$ be a grid.
\begin{itemize}
	\item If $m=4$, then $S_{SC} (G) = [1, nm-4] \setminus \{ 2, 3, 5 \}$.
	\item If $m=5$, then $S_{SC} (G) = [1, nm-5] \setminus \{ 2, 3, 5 \}$.
	\item If $m \ge 6$, then $S_{SC} (G) = [1, nm-6] \setminus \{ 2, 3, 5 \}$.
\end{itemize}
\end{theorem}

\begin{proof}
The excluded values are a consequence of Lemma \ref{forbid}.   Corollary \ref{con=1} shows $1 \in S_{SC} (G)$.   Let $r$ be an integer $4 \le r \le nm-i$, for $i \in \{ 4, 5, 6 \}$.   The three cases are dealt similarly.

If $nm-2n+1 \le r$, then Lemma \ref{n*m-c=nm-k} implies that $r \in S_{SC} (G)$.   Otherwise $r \le nm-2n \le (n-1)(m-1)$, and it follows from Lemmas \ref{n*m-c=4}, \ref{n*m-c=ab}, \ref{n*m-c=nb} and \ref{n*m-c=ab-k-l}, that $r \in S_{SC} (G)$.

\end{proof}

As a final remark, the reader might have noticed that the proofs are
lengthy and technical, which makes them hard to follow.   It would be
a good problem to find a short proof for the main theorem of this article.

\end{document}